\documentclass[a4paper,10pt]{article}
\usepackage[vlined]{algorithm2e}

\SetFuncSty{textsc}
\SetKwFunction{frun}{Run}
\SetKwHangingKw{arun}{\frun}
\SetKwFunction{fset}{Set}
\SetKwHangingKw{aset}{\fset}
\SetKwFunction{fselect}{Select}
\SetKwHangingKw{aselect}{\fselect}
\SetKwFunction{fcompute}{Compute}
\SetKwHangingKw{acompute}{\fcompute}
\SetKwFunction{fsolve}{Solve}
\SetKwHangingKw{asolve}{\fsolve}
\SetKwFunction{festimate}{Estimate}
\SetKwHangingKw{aestimate}{\festimate}
\SetKwFunction{fmark}{Mark}
\SetKwHangingKw{amark}{\fmark}
\SetKwFunction{frefine}{Refine}
\SetKwHangingKw{arefine}{\frefine}
\SetKwFunction{compute}{Compute}
\SetKwFunction{set}{Set}

{\algorithm}%
{\endalgorithm}

\title{Discontinuous Galerkin Finite Element Methods for the Landau-de Gennes Minimization Problem of Liquid Crystals }
\author{Ruma Rani Maity\footnote{
		Department of Mathematics, Indian Institute of Technology Bombay, Powai, Mumbai 400076, India. Email. ruma@math.iitb.ac.in} $\;\;$ Apala Majumdar* \footnote{* Department of Mathematical Sciences, University of Bath, Claverton Down, Bath BA2 7AY, United Kingdom. Visiting Professor,  Indian Institute of Technology Bombay, Powai, Mumbai 400076, India. Email. A.Majumdar@bath.ac.uk}        
	$\;\;$ Neela Nataraj\footnote{Department of Mathematics, Indian Institute of Technology Bombay, Powai, Mumbai 400076, India. Email. neela@math.iitb.ac.in}
}

\usepackage{amsmath,amsthm,amssymb,enumerate, amsbsy}
\usepackage[T1]{fontenc}
\usepackage{gensymb}
\usepackage[square,sort&compress,comma,numbers]{natbib}
\usepackage[sc]{mathpazo}
\usepackage{multirow} 
\usepackage{newtxtext,newtxmath}
\usepackage{subfig}
\usepackage{graphicx}
\usepackage{epstopdf}
\usepackage{cancel}
\usepackage[margin=3cm]{geometry}
\usepackage{tikz}
\usepackage[titletoc]{appendix}
\usepackage{caption}
\usetikzlibrary{shapes,calc}
\usepackage{verbatim}
\usepackage{mathrsfs}
\usepackage{accents}
\usepackage[utf8]{inputenc}
\usepackage{float}
\restylefloat{table}

\usetikzlibrary{decorations.pathmorphing}
\usetikzlibrary{decorations.pathreplacing}
\usetikzlibrary{positioning}
\usetikzlibrary{shapes}
\usetikzlibrary{arrows}
\usetikzlibrary{patterns}
\usetikzlibrary{fadings}
\usetikzlibrary{plotmarks}
\usetikzlibrary{calc}
\usetikzlibrary{intersections}
\tikzstyle{every picture}+=[font=\footnotesize]
\usepackage{paralist}
\usepackage{empheq}
\usepackage{bbm}
\usepackage{latexsym}           
\usepackage{enumerate}
\usepackage{enumitem}
\usepackage{algorithm2e}

\setlist{noitemsep, topsep=0.8ex, partopsep=0pt
	, leftmargin=3em}
\setlist[1]{labelindent=\parindent}

\newlist{axioms}{enumerate}{1}
\setlist[axioms]{font=\bfseries}

\newlist{alphenum}{enumerate}{1}
\setlist[alphenum]{label=\textbf{(\alph*)}, leftmargin=4em}

\newlist{alphienum}{enumerate}{1}
\setlist[alphienum]{label=\textit{(\alph*)}}

\newlist{romanenum}{enumerate}{1}
\setlist[romanenum]{label=\textit{(\roman*)}}

\newlist{romaninenum}{enumerate*}{1}
\setlist[romaninenum]{label=\textit{(\roman*)}}
\usepackage[vlined]{algorithm2e}
\SetKwIF{If}{ElseIf}{Else}{if}{}{else if}{else}{endif}
\SetKwFor{For}{for}{}{endfor}
\usepackage[noabbrev, capitalise]{cleveref}
\usepackage{tabu}
\crefname{equation}{\unskip}{\unskip}
\creflabelformat{equation}{#2(#1)#3}
\newtheorem{thm}{Theorem}[section]

\newtheorem{lem}[thm]{Lemma}

\theoremstyle{definition}

\theoremstyle{remark}

\newtheorem{rem}[thm]{Remark}

\newtheorem{example}{\bf Example}[subsection]

\numberwithin{equation}{section}

\newcommand{\e}{\mathbf{e}}
\newcommand{\R}{\mathbb{R}}

\newcommand{\V}{\mathbf{V}}
\newcommand{\X}{\mathbf{X}}
\newcommand{\h}{\mathbf{H}}

\newcommand{\abs}[1]{\left\lvert#1\right\rvert}
\newcommand{\norm}[1]{\left\lVert#1\right\rVert}
\newcommand{\dx}{\,\textrm{dx}}
\newcommand{\ds}{\,\textrm{ds}}

\newcommand{\E}{\textrm{E}}

\newcommand{\dg}{\textrm{dG}}

\newcommand{\vertiii}[1]{{\vert\kern-0.25ex\vert\kern-0.25ex\vert #1 
		\vert\kern-0.25ex \vert\kern-0.25ex \vert}}
\newcommand{\vertiiinc}[1]{{\vert\kern-0.25ex \vert\kern-0.25ex \vert #1 
		\vert\kern-0.25ex \vert\kern-0.25ex \vert}_{\textrm{NC}}}
\newcommand{\vertiiidg}[1]{{\vert\kern-0.25ex \vert\kern-0.25ex \vert #1 
		\vert\kern-0.25ex \vert\kern-0.25ex\vert}_{\textrm{dG}}}
\newcommand{\dual}[1]{\langle #1 \rangle}
\theoremstyle{plain}

\begin{document}
	
	\maketitle
	
	\begin{abstract}
		We consider a system of second order non-linear elliptic partial differential equations that models the equilibrium configurations of a two dimensional planar bistable  nematic liquid crystal device. Discontinuous Galerkin finite element methods are used to approximate the solutions of this nonlinear problem with non-homogeneous Dirichlet boundary conditions. A discrete inf-sup condition  demonstrates the stability of the discontinuous Galerkin discretization of a well-posed linear problem. We then establish the existence and local uniqueness of the discrete solution of the non-linear problem.
		An a priori error estimates in the energy and $\mathbf{L}^2$ norm are derived and a best approximation property is demonstrated. Further, we prove the quadratic convergence of Newton's iterates along with complementary numerical experiments.
	\end{abstract}
	
	\section{Introduction}
	Liquid crystals are a transitional phase of matter between the liquid and crystalline phases. They
	inherit versatile properties of both liquid phase (e.g. fluidity) and, crystalline phase (optical, electrical,
	magnetic anisotropy) which make them ubiquitous in practical applications ranging from wristwatch
	and computer displays, nanoparticle organizations to proteins and cell membranes.
	Liquid crystals present different phases as temperature varies. We consider the nematic phase formed by rod-like molecules that self-assemble into an ordered structure, such that the molecules tend to align along a preferred orientation. 
	There are three main mathematical models for nematic liquid crystals studied in literature which are the Oseen-Frank model \cite{Oseen1,Ball,energyminimizatnoseenfrank},  Ericksen model \cite{Erickson} and  Landau-De Gennes model \cite{Landau1, Majumdar2010, FiniteelementanalysisLDG}. 
	\medskip 
	
	We present some rigorous results for the discontinuous Galerkin (dG) formulation of a reduced two-dimensional Landau-de Gennes model, following the model problem studied in \cite{MultistabilityApalachong,Tsakonas}.
In the reduced Landau-de Gennes theory, the state of the nematic liquid crystal is described by a tensor order parameter,
	 $\displaystyle \textit{Q}:= s(2\mathbf{n}\otimes\mathbf{n}-\textit{I}),$
	where  $ \textit{Q}\in \mathbf{S}_0:=\lbrace \textit{Q}=(\textit{Q}_{ij})_{1\leq i,j \leq 2} \in \R^{2 \times 2}: \textit{Q}=\textit{Q}^T, \text{tr}\textit{Q}=0 \rbrace $ and $I$ is the $2 \times 2$ identity matrix. We refer to $'\mathbf{n}'$ as the director or the locally preferred in-plane alignment direction of the nematic molecules; $'s'$ is the scalar order parameter that measures the degree of order about $'\mathbf{n}'.$ We can write $'\mathbf{n}'$ as $\mathbf{n}=(\cos\theta,\sin\theta) $, where $\theta$ is the director angle in the plane. The general Landau-de Gennes $\textit{Q} $-tensor order parameter is a symmetric traceless $3\times 3$ matrix and we can rigorously justify our $2$-D approach in certain model situations (see \cite{wang_canevari_majumdar_SIAM2019}).
	\medskip 
	
	 Since $\textit{Q}$ is a symmetric, traceless $2\times 2$ matrix, we can represent $\textit{Q}$ as a two dimensional quantity, $ \Psi = (u,v)$ where $u=s \cos2\theta$ and $v=s \sin2\theta$. The stable nematic equilibria are minimizers of the Landau-de Gennes energy functional given by $\mathcal{E}(\textit{Q}):= \mathcal{E}_{B}(\textit{Q}) + \mathcal{E}_{E}(\textit{Q})+ \mathcal{E}_{S}(\textit{Q}) -\mathcal{E}_{L}(\textit{Q})$
	 	subject to appropriate boundary conditions. The bulk energy over the domain $\Omega$ is  given by  $\mathcal{E}_{B}(\textit{Q}):=\int_\Omega ( -\frac{\alpha^2}{2}\text{tr}{\textit{Q}}^2-\frac{b^2}{3}\text{tr}{\textit{Q}}^3+\frac{c^2}{4}(\text{tr}{\textit{Q}}^2)^2 ) \dx,$
	 	$ \alpha$, $b$ and $c$ being temperature and material dependent constants. 
	 	The one-constant elastic energy is  $\mathcal{E}_{E}(\textit{Q}):=\int_\Omega \frac{L_{el}}{2}\abs{\nabla\textit{Q}}^2 \dx$ with $  L_{el}>0 $ being an elastic constant.
	 	The surface anchoring energy is  $\mathcal{E}_{S}(\textit{Q}):=\int_{\partial\Omega} W \abs{ (\textit{Q}_{11},\textit{Q}_{12}) - \mathbf{g} }^2 \, \textrm{ds}$
	 	with the anchoring strength $ W$  on $\partial \Omega$ and  prescribed Lipschitz continuous boundary function $ \mathbf{g}:\partial\Omega \rightarrow \mathbb{R}^2$.
	  The electrostatic energy  $\mathcal{E}_{L}(\textit{Q}):=\int_{\Omega}-C_0 (\textit{Q}\textit{E}) \cdot \textit{E} \dx$
	with $C_0$ depending on vaccum permittivity and material dependent constants
	and $\textit{E}$ being the electric field vector. Here '$\cdot $' denotes the scalar product of vectors of  function entries.
	\medskip 
	
	In the absence of surface effects ($\mathcal{E}_{S}(\textit{Q})$) and external fields ($\mathcal{E}_{L}(\textit{Q})$), the Landau-de Gennes energy functional in the dimensionless form considered in this article is  given by \cite{MultistabilityApalachong}
	\begin{align} \label{eq:2.1.1}
	\mathcal{E}(\Psi_\epsilon) =\int_\Omega (\abs{\nabla \Psi_\epsilon}^2 + 
	\epsilon^{-2}(\abs{\Psi_\epsilon}^2  - 1)^2) \dx, 
	\end{align}
	where $\Psi_\epsilon=(u,v)\in \mathbf{H}^1(\Omega):=H^1(\Omega) \times H^1(\Omega)$ and $\epsilon$ is a small positive parameter that depends on  the elastic constant, bulk energy parameters and the size of the domain. More precisely, (see \cite{MultistabilityApalachong}) after rescaling and suitable non-dimensionalization, $\epsilon = \sqrt{\frac{L_{el}}{L^2c^2}} $ is proportional to the ratio of the nematic correlation length to a characteristic domain size. The nematic correlation length (usually on the scale of nanometers) is typically related to defect core sizes. The  $\epsilon \rightarrow 0$ limit describes the macroscopic limit where the domain size is much larger than the correlation length i.e. describes micron scale geometries or even larger geometries. Some typical values \cite{MultistabilityApalachong} of these physical parameters are $L_{el}= 10^{-11} \,\text{N m} $, $L= 8 \times 10^{-5} \text{ m}$  and $c^2=1 \times 10^6\, \text{N  m}^{-1}$. We are concerned with the minimization of the functional $\mathcal{E}$ for the Lipschitz continuous boundary function $ \mathbf{g}:\partial\Omega \rightarrow \mathbb{R}^2$. More precisely, the admissible space is 
	$\X=\left \{ \mathbf{w} \in \mathbf{H}^1(\Omega) :
	\mathbf{w}= \mathbf{g} \,\, \text{on} \,\, \partial \Omega   \right \}.
	$
	The strong formulation of \eqref{eq:2.1.1} seeks $\Psi_\epsilon \in \mathbf{H}^1(\Omega) $ such that 
	\begin{align}  \label{2.3.1.2} 
	-\Delta \Psi_\epsilon =2\epsilon^{-2}(1-\abs{\Psi_\epsilon}^2)\Psi_\epsilon \text{ in } \Omega \,\, \text{ and }\,\, \Psi_\epsilon = \mathbf{g} \text{ on } \partial \Omega. 
	\end{align}
	The behavior of $\Psi_\epsilon$ as $\epsilon \rightarrow 0 $ has been studied by Brezis et al in \cite{Bethuel} with the assumption $0<\epsilon <1$ and an $\epsilon$ independent bound for $\vertiii{\Psi_\epsilon}_{\mathbf{H}^2}$ has been established. The minimizers of the energy functional \eqref{eq:2.1.1} for a given smooth boundary data $\mathbf{g}$ with $\abs{\mathbf{g}}=1 \text{ on } \partial\Omega$ for a  smooth bounded domain has been well-studied \cite{Bethuel}. Further, the authors rigorously prove in \cite{Bethuel} that $\Psi_\epsilon \rightarrow \Psi_0$ as $\epsilon \rightarrow 0 $ in $C^{1,\alpha}(\bar{\Omega})$ \: $ \text{ for all }\alpha <1 $, where $\Psi_0 $ is a harmonic map and is a solution of 
	$-\Delta \Psi_0= \Psi_0 \abs{\nabla \Psi_0}^2 \text{ on } \Omega,$
	$\abs{ \Psi_0}= 1 $  { on } $\Omega,$ \;
	$\Psi_0 =  \mathbf{g}$  {on}  $\partial \Omega$.
We are interested in the dG finite element approximation of regular solutions of the boundary value problem \eqref{2.3.1.2}.
Let  ${\rm N}: \mathbf{H}^1(\Omega) \rightarrow \mathbf{H}^1(\Omega)^*$ be 
the Ginzburg-Landau operator  defined as 
${\rm N}(\Psi_\epsilon):=-\Delta \Psi_\epsilon +2\epsilon^{-2}(\abs{\Psi_\epsilon}^2-1)\Psi_\epsilon$. 
  For $\epsilon$ small enough, the linearized operator $ \displaystyle L_\epsilon:={\rm DN}({\Psi_\epsilon})$ is bijective when defined between standard spaces (e.g., $L_\epsilon: \mathbf{W}^{2,2} \rightarrow \mathbf{L}^2 $) \cite{pacard2000linear},  although the norm of its inverse blows up as $\epsilon \rightarrow 0$.
	\medskip
		
	This model problem has been studied using the conforming finite element method for \eqref{2.3.1.2}  for a {\it fixed } $\epsilon$  in \cite{MultistabilityApalachong} on a square domain with $\mathbf{g}$ consistent tangent boundary conditions. The variational formulation of a more generic Landau-de Gennes model with homogeneous Dirichlet boundary data has been studied \cite{FiniteelementanalysisLDG} in which an  abstract approach of the finite element approximation of nonsingular solution branches has been analyzed in the conforming finite element set up. However the analysis for the non-homogeneous boundary conditions has not been considered in this work. In  \cite{energyminimizatnoseenfrank} and \cite{ericksonmixed}, the authors have discussed a mixed finite element method  for the  Frank-Oseen and Ericksen-Leslie models for nematic liquid crystals, respectively. \medskip 
	
	It is well-known that Allen-Cahn equation \cite{allencahn} is the gradient flow equation associated with the Lyapunov energy functional 
	$\mathcal{J}_\epsilon(u)=\frac{1}{2} \int_{\Omega}(\abs{\nabla u }^2 +  \frac{\epsilon^{-2}}{2}(\abs{u}^2-1)^2)\dx,$ where $u:\Omega \rightarrow R$ is a  scalar valued function and $\epsilon$ is a small parameter (not to be confused with the $\epsilon$ in \eqref{2.3.1.2})
	known as an “interaction length” which is small compared to the characteristic dimensions on
	the laboratory scale. 
	 Numerical approximations of the Allen-Cahn equation have been extensively investigated
	in the literature. \textit{A priori} error estimate for the error bounds as a function of $\epsilon$ have been analyzed and shown to be of polynomial order in $\epsilon^{-1}$  by Feng and Prohl in \cite{Fengprohl}, for the conforming finite element approximation of the Allen-Cahn problem. A symmetric interior penalty discontinuous Galerkin method \cite{FengLi}  and a posteriori error analysis \cite{Feng2005, Bartels} (which has a low order polynomial dependence on  $\epsilon^{-1}$)
	 have also been studied  for Allen-Cahn equation. A dG scheme has been proposed recently in \cite{Antonopoulou} for the  $\epsilon$- dependent stochastic Allen–Cahn equation with mild space-time noise posed on a bounded domain of $\mathbb{R}^2$. The Allen Cahn work is relevant to time-dependent front propagation in nematic liquid crystals, in certain reduced symmetric situations \cite{SPICER_MAJUMDAR_MILEWSKI_2016}.
	 The problem considered in this article is different from the Allen Cahn equation; we have a system of two coupled nonlinear partial differential equations (PDEs) in a time independent scenario. 
	 \medskip 
	 
	 Our motivation comes from the planar bistable nematic device reported in \cite{MultistabilityApalachong}. This device consists of a periodic array of shallow square or rectangular wells, filled with nematic liquid crystals, subject to tangent boundary conditions on the lateral surfaces. The vertical well height is much smaller than the cross-sectional dimensions and hence, it is reasonable to assume invariance in the vertical direction and to model the profile on the bottom cross-section, taking the domain to be a square as opposed to a three-dimensional square well. This model reduction can be rigorously justified using gamma-convergence techniques \cite{golovaty2015, wang_canevari_majumdar_SIAM2019}. The tangent boundary conditions require that the nematic director, identified with the leading eigenvector $\mathbf{n}$ of the Landau-de Gennes $\textit{Q}$-tensor order parameter, lies in the plane of the square and $\mathbf{n}$ is tangent to the square edges. Indeed, this motivates the choice of the Dirichlet conditions for $\textit{Q}$ on the square edges as described below. The tangent boundary conditions can also be phrased in terms of surface anchoring energies but this results in mixed boundary-value problems for $\textit{Q}$ which are relatively more difficult to analyse than the Dirichlet counterparts. Our results can be applied to the planar bistable nematic device; we can model a single well as a square domain with Dirichlet tangent boundary conditions on the square edges and analyze discontinuous Galerkin finite element methods (dGFEMs) to approximate regular solutions of \eqref{2.3.1.2} for a fixed $\epsilon$.
	 This involves a semilinear system of PDEs with a \textit{cubic nonlinearity} (see \eqref{2.3.1.2}) and \textit{non-homogeneous boundary conditions}. As reported in \cite{MultistabilityApalachong, Tsakonas}, there are six experimentally observed stable nematic equilibria, labeled as diagonal and rotated states, for this model problem. The nematic director roughly aligns along one of the square diagonals in the diagonal states whereas the director rotates by $\pi$ radians between a pair of opposite parallel edges, in a rotated state. There are two diagonal states, since there are two square diagonals, and four rotated states related to each other by a $\frac{\pi}{2}$ rotation. 
	The dGFEMs are attractive because they are element-wise
	conservative, are flexible with respect to local mesh adaptivity, are easier to implement than finite volume
	schemes, allow for non-uniform degrees of approximations for solutions with variable regularity over the computational domain and can handle non-homogeneous boundary condition in a natural way. These methods also relax the inter-element continuity requirement in conforming FEM. An {\it a priori}  error analysis of dGFEMs for general
	elliptic problems has been derived in \cite{Prudhomme,Gudi2008,Gudi2010}. For a comprehensive study of several dGFEMs applied to elliptic problems, see \cite{ArnoldDouglasBrezzi}. The dGFEMs are also well studied for fourth order elliptic problems \cite{Gaurang,Houston}. Recently, dGFEMs have been studied for the von Kármán equations \cite{Gaurang} that involves a quadratic non-linearity and homogeneous boundary conditions.
	\medskip

	To the best of our knowledge, dGFEMs have not been analysed for the nonlinear system derived from the reduced two-dimensional Landau-de Gennes energy in \eqref{eq:2.1.1} and this is the primary motivation for our study.
	Our contributions can be summarized as follows.
	\begin{enumerate}
		\item We derive an elegant representation of the nonlinear operator, convergence analysis with $h $-$\epsilon $\textit{ dependency} and \textit{a priori} error estimate of a dGFEM formulation for the system \eqref{2.3.1.2} with \textit{non homogeneous boundary conditions}. The choice of the discretization parameter $h$ that depends on $\epsilon$ ensures that (i)
			the dG discretization of a linearized problem is well-posed and (ii) the corresponding discrete non-linear problem has a unique solution following applications of contraction mapping theorem.
		\item We prove a best approximation result for  regular solutions of the non-linear problem \eqref{2.3.1.2}.
		\item We prove the quadratic convergence of the Newton's iterates to the approximate dGFEM solution.
		\item Our numerical results confirm the theoretical orders of convergence and rate of convergence as a function of $h$ and $ \epsilon$, in the context of the planar bistable nematic device and other representative examples.
	\end{enumerate}
	\medskip
	
	Throughout the paper, standard notations on Sobolev spaces and their norms are employed. The standard semi-norm and norm on $H^s(\Omega)$ $(\text{resp.} \,W^{s,p}(\Omega))$ for $s,p$ positive real numbers, are denoted by $\abs{\cdot}_s$ and $\norm{\cdot}_s$ $(\text{resp. } \abs{\cdot}_{s,p} \text{ and } \norm{\cdot}_{s,p} )$. The standard $L^2(\Omega)$ inner product is denoted by $(\cdot,\cdot)$. We use the notation $\mathbf{H}^s(\Omega)$ (resp.\,$\mathbf{L}^p(\Omega)$) to denote the product space $H^s(\Omega) \times H^s(\Omega)$ $(\text{resp. } L^p(\Omega) \times L^p(\Omega))$. The standard norms $\vertiii{\cdot}_s$ ($\text{resp. }\vertiii{\cdot}_{s,p}$) in the Sobolev spaces $\mathbf{H}^s(\Omega)$ ($\text{resp. }\mathbf{W}^{s,p}(\Omega)$) defined by $\vertiii{\Phi}_s\!=(\norm{\phi_1}_s^2+\norm{\phi_2}_s^2)^{\frac{1}{2}}$ for all $ \Phi\!=\!(\phi_1, \phi_2) \in \!\mathbf{H}^s(\Omega) $ $ ( \text{resp.  }\vertiii{\Phi}_{s,p}\!\!=(\norm{\phi_1}_{s,p}^2+\norm{\phi_2}_{s,p}^2)^{\frac{1}{2}}$ for all $\Phi\!=\!(\phi_1, \phi_2) \!\in \!\mathbf{W}^{s,p}(\Omega)).$
	The norm on $\mathbf{L}^2(\Omega)$ space is defined by $\vertiii{\Phi}_0\!=(\norm{\phi_1}_0^2+\norm{\phi_2}_0^2)^{\frac{1}{2}}$ for all $ \Phi\!=\!(\phi_1, \phi_2) \in \!\mathbf{L}^2(\Omega). $
	Set $\textrm{V}:=H_0^1(\Omega)= \left\{\phi \in L^2(\Omega): \frac{\partial \phi}{\partial x}, \frac{\partial \phi}{\partial y}
	\in L^2(\Omega)\,\, , \phi|_{\partial {\Omega}}=0 \right\}$ and $ \V\!= \,\!\mathbf{H}_0^1(\Omega)=H_0^1(\Omega) \times H_0^1(\Omega)$. The inequality $ a \lesssim b $ abbreviates $ a \leq Cb$ with the constant $C>0$ independent of mesh-size parameter $'h'$ and $'\epsilon'$. The constants that appear in various Sobolev imbedding results in the sequel are denoted using a generic notation $C_S$.  
	\medskip 
	
	The paper is organised as follows. In Section \ref{preliminaries}, we present the model problem along with the weak formulation and some preliminary results. We state the dG finite element formulation of the problem in Subsection \ref{dg_formulation} and our main results are stated in Subsection \ref{section 3.1}. The existence and uniqueness of the discrete solution of the non-linear problem, error estimates, best approximation result and the convergence of Newton's method are presented as main theorems. Section \ref{auxiliary results} contains some auxiliary results needed to prove the main results. A discrete inf-sup condition for a discrete bilinear form has been established. In Section \ref{main result}, we prove the main theorems. A contraction map has been defined on the discrete space to use a fixed point argument for proving the existence and uniqueness of discrete solution. An alternative proof of the existence and uniqueness of the solution of the discrete problem using Newton-Kantorovich theorem has been given in this section. This is followed by a proof of the quadratic convergence of Newton's method and numerical results that are consistent with the theoretical results in Section \ref{numerical}.
	\section{ Preliminaries } \label{preliminaries} 
	In this section, we introduce the weak formulation for \eqref{2.3.1.2}  and establish some boundedness results. The details of derivation of the weak formulation in presented in \ref{Weak_formulation}.\\
In the weak formulation of \eqref{2.3.1.2},  we seek $\Psi_\epsilon \in \mathbf{X}$ such that
	\begin{align} \label{eq:2.1.5}
	{\rm N}(\Psi_\epsilon;\Phi):=A(\Psi_\epsilon,\Phi)+B(\Psi_\epsilon,\Psi_\epsilon,\Psi_\epsilon,\Phi)+C(\Psi_\epsilon,\Phi)=0 \,\,\,\, \text{ for all } \Phi \in \V,
	\end{align}
	where for all $ \Xi=(\xi_1,\xi_2), \boldsymbol \eta= (\eta_1, \eta_2), \Theta=(\theta_1,\theta_2), \Phi=(\phi_1,\phi_2)  \in  \mathbf{H}^1(\Omega) $  and $C_\epsilon=2\epsilon^{-2},$
	\begin{align}
	&A(\Theta,\Phi):=a(\theta_1,\phi_1)+ a(\theta_2,\phi_2),\,\,
	C(\Theta,\Phi):=c(\theta_1,\phi_1)+ c(\theta_2,\phi_2),\,\,\, \label{eq:2.1.6}\\&
	B(\Xi,\boldsymbol \eta,\Theta,\Phi):=\frac{C_\epsilon}{3}\int_\Omega \left((\Xi \cdot \boldsymbol \eta)(\Theta \cdot \Phi)+2(\Xi \cdot \Theta)(\boldsymbol \eta \cdot \Phi)\right) \dx =\frac{1}{3}(3b(\xi_1,\eta_1,\theta_1,\phi_1)+3b(\xi_2,\eta_2,\theta_2,\phi_2) \notag \\& 
	\qquad \qquad \qquad \qquad + 2b(\xi_2,\eta_1,\theta_2,\phi_1)+ 2 b(\xi_1,\eta_2,\theta_1,\phi_2)+ b(\xi_2,\eta_2,\theta_1,\phi_1)+b(\xi_1,\eta_1,\theta_2,\phi_2)),\label{eq:2.1.7}
	\\&
 \text{and for } \xi,\eta,\theta,\phi \in H^1(\Omega), 
	\:  a(\theta,\phi):=\int_{\Omega}  \nabla \theta \cdot \nabla \phi \dx,  \: b(\xi,\eta,\theta,\phi):= C_\epsilon\int_\Omega \xi\eta\theta \phi \dx, \nonumber \\ 
	& \text{ and }  c(\theta,\phi):=-C_\epsilon \int_\Omega \theta \phi \dx. \notag
	\end{align}
	\begin{rem} 
		When $\Xi=\boldsymbol \eta = \Theta= \Psi_\epsilon:=(u,v)$,
		\begin{align*}
		B(\Psi_\epsilon,\Psi_\epsilon,\Psi_\epsilon,\Phi) = b(u,u,u,\phi_1)+b(v,v,u,\phi_1)+b(v,v,v,\phi_2)+b(u,u,v,\phi_2).
		\end{align*}
		The operator $B(\cdot, \cdot, \cdot,\cdot)$ corresponds to the non-linear part of the system \eqref{2.3.1.2}. Such representation of $B(\cdot, \cdot, \cdot,\cdot)$ yields nice properties proven in Lemma \ref{2.6} which makes the analysis elegant. 	
	\end{rem}
	\begin{lem}\label{4.2}\emph{(Poincar\'e inequality)}\cite{KesavaTopicsFunctinal}
	Let $\Omega$ be a bounded open Lipschitz domain in $\mathbb{R}^2.$ Then there exists a positive constant $\alpha_0 =\alpha_0(\Omega)$ such that 
	\begin{align*}
	\alpha_0 \norm{\phi}_{L^2(\Omega)} \leq \abs{\phi}_{1,\Omega} \,\,\text{ for all } \phi \in H^1_0(\Omega).
	\end{align*}
    \end{lem}
	\noindent The boundedness and coercivity of $A(\cdot,\cdot)$ and the boundedness of $C(\cdot,\cdot)$ given below can be easily verified. For all $\Theta$, $\Phi \in \V$,
	\begin{align}
	&A(\Theta,\Phi)\leq \vertiii{\Theta}_1 \vertiii{\Phi}_1, \quad  A(\Theta,\Theta) \geq C_{\alpha_0} \vertiii{\Theta}_1^2, \label{2.3}\\&
 \text{ and }	C(\Theta,\Phi)\leq  C_\epsilon \vertiii{\Theta}_1 \vertiii{\Phi}_1, \quad \quad \quad \,\, \label{eq:2.1.5.4}
	\end{align}
where $C_{\alpha_0}$ depends on $\alpha_0.$	
	The next lemma establishes two boundedness results for $B(\cdot,\cdot,\cdot,\cdot)$.
	\begin{lem} \emph{(Boundedness of $B(\cdot,\cdot,\cdot, \cdot) $ )}\label{2.6}
		For all  $\Xi$, $\boldsymbol \eta$, $\Theta$, $\Phi \in \V$, 
		\begin{align} \label{eq:2.1.5.6}
		B(\Xi,\boldsymbol \eta,\Theta,\Phi)\lesssim \epsilon^{-2} \vertiii{\Xi}_1\vertiii{\boldsymbol \eta}_1 \vertiii{\Theta}_1 \vertiii{\Phi}_1,
		\end{align} 
		and for all  $\Xi , \boldsymbol \eta \in \mathbf{H}^2(\Omega)$, $\Theta$, $\Phi \in \V$,
		\begin{align}\label{eq:2.1.5.3}
		B(\Xi,\boldsymbol \eta,\Theta,\Phi)\lesssim \epsilon^{-2}  \vertiii{\Xi}_2\vertiii{\boldsymbol \eta}_2\vertiii{\Theta}_0 \vertiii{\Phi}_0.
		\end{align}
	\end{lem}
	\begin{proof}
		It is enough to prove the results for $b(\cdot,\cdot,\cdot,\cdot)$;
		 then \eqref{eq:2.1.5.6} and \eqref{eq:2.1.5.3} follow from the definition of $B(\cdot,\cdot,\cdot,\cdot)$ and a grouping of the terms. 
		For $\xi,\eta, \theta, \phi \in H^1(\Omega)$, a use of H\"older's inequality and the Sobolev imbedding result ${H}^1(\Omega) \hookrightarrow  L^4(\Omega) $ \cite{Evance19} leads to \eqref{eq:2.1.5.6} as
		\begin{align*}
		b(\xi,\eta, \theta, \phi)
		\leq C_\epsilon
		\norm{\xi}_{0,4} \norm{\eta}_{0,4} \norm{\theta}_{0,4} \norm{\phi}_{0,4} \lesssim \epsilon^{-2} \norm{\xi}_1 \norm{\eta}_1 \norm{\theta}_1 \norm{\phi}_1. 
		\end{align*}	
		 For $\xi$, $\eta \in {H}^2(\Omega)$, a use of the 
		Sobolev imbedding result ${H}^2(\Omega)\hookrightarrow L^\infty(\Omega)$ and the Cauchy-Schwarz inequality leads to  
		\begin{align*}
		b(\xi,\eta, \theta, \phi)\lesssim \epsilon^{-2} \norm{\xi}_{0,\infty}  \norm
		{\eta}_{0,\infty}  \norm{\theta}_0 \norm{\phi}_0 \lesssim  \epsilon^{-2} \norm{\xi}_2  \norm
		{\eta}_2  \norm{\theta}_0 \norm{\phi}_0 \,\, \text{for all }  \theta, \phi \in \textrm{V},
		\end{align*} 
		where $"\lesssim"$ in the last two inequalities above absorbs  $C_S$ and the constant from  $C_\epsilon$.
	\end{proof}	
 The existence of minimizers of \eqref{eq:2.1.1} follows from the coercivity of $\mathcal{E}$ and its convexity in  $\nabla \Psi_\epsilon$ in \eqref{eq:2.1.1} and this implies the existence of a solution of the non-linear system in \eqref{eq:2.1.5}. The regularity result in the next lemma follows from arguments in \cite{FiniteelementanalysisLDG, grisvard} and in detailed in \ref{Regularity result}.
	\begin{lem} [Regularity result]\label{regularity}
		Let $\Omega$ be an open, bounded, Lipschitz and convex domain of $\mathbb{R}^2.$ Then for $\mathbf{g} \in \mathbf{H}^{\frac{3}{2}}(\partial \Omega),$ any solution of \eqref{2.3.1.2} belongs to  $\h^{2}(\Omega)$.
	\end{lem}
	In this article, we approximate regular solutions \cite{keller} $\Psi_\epsilon$ of \eqref{2.3.1.2} for a given $\epsilon$. The regularity of solution implies that the linearized operator $\dual{{\rm DN}(\Psi_\epsilon)\cdot, \cdot}$ is invertible in the Banach space and is equivalent to the following inf-sup condition \cite{Ern} 
	\begin{align}\label{2.9}
	0< \beta := \inf_{\substack{\Theta \in \V \\  \vertiii{\Theta}_1=1}} \sup_{\substack{\Phi \in \V \\  \vertiii{\Phi}_1=1}}\dual{{\rm DN}(\Psi_\epsilon)\Theta, \Phi}=  \inf_{\substack{\Phi \in \V \\  \vertiii{\Phi}_1=1}} \sup_{\substack{\Theta \in \V \\  \vertiii{\Theta}_1=1}}\dual{{\rm DN}(\Psi_\epsilon)\Theta, \Phi},
	\end{align}
	where $\dual{{\rm DN}(\Psi_\epsilon)\Theta, \Phi}:=A(\Theta,\Phi)+3B(\Psi_\epsilon,\Psi_\epsilon,\Theta,\Phi)+C(\Theta,\Phi)$ and  the inf-sup constant $\beta$ depends on $\epsilon$.
From now onwards, the subscript $\epsilon$ in $\Psi_\epsilon$ is suppressed in the sequel for notational brevity. 
	\section{Discrete formulation} \label{discrete}
	In this section, we derive the dGFEM formulation for \eqref{2.3.1.2} and state our main results.
	\subsection{The dGFEM formulation}\label{dg_formulation}
	Let $\mathcal{T}$ be a triangulation \cite{ciarlet} of $\bar{\Omega}$ into triangles and let the discretization parameter $h$ associated with the partition $\mathcal{T}$ be defined as $h= \max_{T \in \mathcal{T}} h_T,$ where $h_T= diam(T) $.  Let $\mathcal{E}_i $( resp. $ \mathcal{E}_D$) denote the interior (resp. boundary) edges of $\mathcal{T}$ and  $\mathcal{E}:=\mathcal{E}_i \cup \mathcal{E}_D$. Also, the boundary of an element $T$ is denoted by $\partial T$ and the unit normal vector outward from $T$ is denoted by $n.$ 
	For any interior edge $E$ shared by two triangles $T^+$ and $T^-$, let the unit normal pointing from  $T^+$ to $T^-$ be $n_+$ [see Figure \ref{unit normal}]. We assume the triangulation $\mathcal{T} $ be shape regular \cite{ciarlet} in the sense that there exists $\rho > 0$ such that if $h_T$ is the diameter of $T$, then $T$ contains a ball of radius $\rho h_T $ in its interior. \medskip
	\begin{figure}[h!] 
		\centering
		\includegraphics[height=4cm, width=6cm]{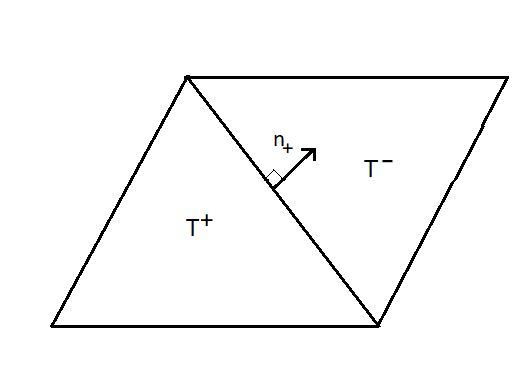}
		\caption{ The neighboring triangles $T^+$ and $T^-$ and unit normal $n_+$.}
		\label{unit normal}
	\end{figure}
	For a positive integer $s$, define the broken Sobolev space     by  $H^s(\mathcal{T}):=\{v \in L^2(\Omega) : v|_T \in H^s(T) \,  \text{ for all } T \in \mathcal{T} \}$ equipped with the broken norm $\norm{v}_{s,\mathcal{T}}^2=\sum_{ T \in \mathcal{T}} \norm{v}_{s,T}^2 .$ 
	Denote $ \h^s(\mathcal{T}) = H^s(\mathcal{T}) \times H^s(\mathcal{T})$ to be the product space with the norm $\vertiii{\Phi}_{s,\mathcal{T}}^2=\norm{\phi_1}_{s,\mathcal{T}}^2+\norm{\phi_2}_{s,\mathcal{T}}^2 \,\, \text{ for all } \Phi=(\phi_1, \phi_2) \in \h^s(\mathcal{T}).$ For $u \in H^s(\mathcal{T}) $, we define the broken gradient $\nabla_\dg u$ by 
	$(\nabla_\dg u)|_T=\nabla_\dg (u|_T) \,\, \text{ for all } T \in \mathcal{T}.$
	Follow the standard convention \cite{Prudhomme} for the jump and average. 
	Consider the finite dimensional space that consists of piecewise linear polynomials defined by $\displaystyle V_h=\{v \in L^2(\Omega): v|_T \in P_1(T) \,\, \text{for all } T \in \mathcal{T}\},$
	and equip it with the mesh dependent norm defined by
	$\displaystyle \norm{v}^2_{\dg}=\abs{v}^2_{\textrm{H}^1(\mathcal{T})}+ J^\sigma(v,v)=\sum_{ T \in \mathcal{T}} \int_T \abs{ \nabla v}^2 \dx + \sum_{E \in \mathcal{E}} \int_{E} \frac{\sigma}{h} [v]^2 \ds,$
	where $\sigma > 0$ is the penalty parameter. Let $\V_h\!: =V_h \times V_h $ be equipped with the product norm defined by $\vertiiidg{\Phi_\dg}^2= \norm{\phi_1}^2_\dg +  \norm{\phi_2}^2_\dg \text{ for all } \Phi_\dg=(\phi_1,\phi_2) \in \V_h.  $
\medskip

	\noindent The discontinuous Galerkin formulation corresponding to \eqref{2.3.1.2} seeks $ \Psi_{\dg}\!\in\!\V_h$ such that for all $ \Phi_{\dg} \in  \V_h,$
	\begin{align}\label{2.3.13}
	{\rm N}_h(\Psi_\dg;\Phi_\dg):=A_{\dg}(\Psi_{\dg},\Phi_{\dg})+B_{\dg}(\Psi_{\dg},\Psi_{\dg},\Psi_{\dg},\Phi_{\dg})+C_{\dg}(\Psi_{\dg},\Phi_{\dg})=L_\dg(\Phi_{\dg}), 
	\end{align}
	where for $ \Xi=(\xi_1,\xi_2),\, \boldsymbol \eta=(\eta_1, \eta_2),\,\Theta=(\theta_1,\theta_2)  ,\, \Phi=(\phi_1,\phi_2)
	\in \h^1(\mathcal{T}) $, $\displaystyle L_\dg(\Phi_{\dg})=l^1_\dg(\phi_1)+l^2_\dg(\phi_2),$
	\begin{align}
	A_{\dg}(\Theta,\Phi)& :=a_{\dg}(\theta_1,\phi_1)+a_{\dg}(\theta_2,\phi_2),  C_{\dg}(\Theta,\Phi):=c_{\dg}(\theta_1,\phi_1)+c_{\dg}(\theta_2,\phi_2) , \notag \\	
	B_\dg(\Xi,\boldsymbol \eta,\Theta,\Phi)& :=\frac{C_\epsilon}{3}\sum_{ T \in \mathcal{T}} \int_T  \left((\Xi \cdot \boldsymbol \eta)(\Theta \cdot \Phi)+2(\Xi \cdot \Theta)(\boldsymbol \eta \cdot \Phi)\right) \dx \notag\\ 
	&= \frac{1}{3}(3b_\dg(\xi_1,\eta_1,\theta_1,\phi_1)+3b_\dg(\xi_2,\eta_2,\theta_2,\phi_2)
	+2b_\dg(\xi_2,\eta_1,\theta_2,\phi_1) + 2 b_\dg(\xi_1,\eta_2,\theta_1,\phi_2) \nonumber \\
	& \qquad \quad 
	+b_\dg(\xi_2,\eta_2,\theta_1,\phi_1)+b_\dg(\xi_1,\eta_1,\theta_2,\phi_2)), \label{3.2}
	\end{align}
	and for 
	$\theta, \phi \in H^1(\mathcal{T})$, $\mathbf{g}=(g_1, g_2)$, $\lambda \in [-1,1],$ 
    $\displaystyle a_{\dg}(\theta,\phi):=a_h(\theta,\phi)-J(\theta, \phi)+\lambda J(\phi,\theta) +J^\sigma(\theta,\phi),$
	\begin{align*}
	&a_h(\theta,\phi):=\sum_{ T \in \mathcal{T}} \int_T \nabla \theta \cdot \nabla \phi \dx, J(\theta,\phi):= \sum_{E \in \mathcal{E}} \int_{E} \{\frac{\partial \theta}{\partial \eta}\} [\phi] \ds, \text{ and } J^\sigma(\theta,\phi):= \sum_{E \in \mathcal{E}} \int_{E} \frac{\sigma}{h} [\theta] [\phi] \ds, \\	
	&b_\dg(\xi, \eta,\theta ,\phi):= C_\epsilon \sum_{ T \in \mathcal{T}}\int_T  \xi \eta \theta \phi \dx ,\,\, c_\dg(\theta,\phi):=-C_\epsilon \sum_{ T \in \mathcal{T}}\int_T \theta \phi \dx,\\&
	\text{  and }\,\,
	l^i_\dg(\phi):=
	\lambda \sum_{E \in \mathcal{E}_D} \int_{E} \frac{\partial \phi}{\partial \eta} g_i \ds+ \sum_{E \in \mathcal{E}_D} \int_{E}\frac{\sigma}{h} g_i \phi \ds \text{ for } 1\leq i \leq 2.
	\end{align*}
	\begin{rem}
		The parameter values $\lambda= -1, 0, 1$ corresponds to symmetric interior penalty, incomplete interior penalty and non-symmetric interior penalty dG methods, respectively in the context of linear problems. 
	\end{rem}
	\subsection{The main results} \label{section 3.1}
	Our main results are given below.
	 The proofs are presented in Sections \ref{main result} and \ref{numerical}.  The details of the suppressed constants in '$\lesssim$' in  the Theorems \ref{2.5.2} - \ref{sipg} will be made clear in the proofs given in Section \ref{main result}.
	\begin{thm}[Existence, uniqueness and dG norm error estimate]\label{2.5.2}
		Let $\Psi$ be a regular solution of the non-linear system
		\eqref{eq:2.1.5}. For a given fixed $\epsilon>0,$ sufficiently large $\sigma$ and sufficiently small discretization parameter chosen as $h=O(\epsilon^{2+\alpha} )$ for any $\alpha > 0$, there exists a unique solution $\Psi_{\rm dG}$ of the discrete non-linear problem \eqref{2.3.13} that approximates $\Psi $ such that 
		\begin{align*}
		\vertiii{\Psi-\Psi_{\rm dG}}_{\rm dG}\lesssim h.
		\end{align*}
	\end{thm}

\begin{thm}[Best approximation result]\label{thmbestapproximation}
		Let $\Psi$ be a regular solution of the non-linear system
	\eqref{eq:2.1.5}.
	 For a given fixed $\epsilon>0,$ sufficiently large $\sigma$ and sufficiently small discretization parameter chosen as $h=O(\epsilon^{2+\alpha} )$ with $\alpha > 0$,  
	 the unique discrete solution $\Psi_{\rm{dG} } $ of \eqref{2.3.13} that approximates $\Psi $ satisfies the best-approximation property 
	\begin{align*}
	\vertiii{\Psi- \Psi_{\rm dG }}_{\rm dG }\lesssim (1+\epsilon^{-2}) \min_{\Theta_{\rm dG} \in \V_h}\vertiii{\Psi- \Theta_{\rm dG}}_{\rm dG}.
	\end{align*}
\end{thm}
\begin{rem}
Theorem \ref{thmbestapproximation} shows that discrete solution $\Psi_{\rm dG }$ is "the best" approximation of $\Psi$ in $\V_{h}$, up to a constant. Best approximation result is mainly motivated by "the best" approximation of discrete solution of linear PDEs in Céa's lemma \cite{ciarlet} for conforming FEM.  
\end{rem}
\noindent We establish the $\mathbf{L}^2  $ norm error estimate when the parameter $\lambda$ that appears in the discrete non-linear system  \eqref{2.3.13} (in the term $A_{\rm dG}(\cdot, \cdot)$ through $a_{\dg}(\cdot, \cdot)$) takes the value $-1$.
\begin{thm}[$\mathbf{L}^2  $ norm error estimate] \label{sipg}
	Let $\Psi$ be a regular solution of the non-linear system
	\eqref{eq:2.1.5}. For a given fixed $\epsilon>0,$ sufficiently large $\sigma$ and sufficiently small discretization parameter chosen as $h=O(\epsilon^{2+\alpha} )$ for  $\alpha > 0$, there exists a unique solution $\Psi_{\rm dG}$ of the discrete non-linear problem \eqref{2.3.13} that approximates $\Psi $  such that 
	\begin{align*}
	\vertiii{\Psi-\Psi_{\rm dG}}_{0}\lesssim h^2(1+ (1+\epsilon^{-2} )^2).
	\end{align*}
\end{thm}
\noindent We use Newton's method \cite{keller} for computation of discrete solutions. It is a standard and very effective root-finding method to approximate the roots of non-linear system of PDEs.
	\begin{thm}[Convergence of Newton's method] \label{2.1.9.1}
		Let $\Psi$ be a regular solution of the non-linear system
		\eqref{eq:2.1.5} and let $\Psi_{\rm{dG} }$ solve \eqref{2.3.13}. 
		For a given fixed $\epsilon>0,$ sufficiently large $\sigma$ and sufficiently small discretization parameter chosen as $h=O(\epsilon^{2+\alpha} )$ with $\alpha>0$,
	 there exists $\rho_1 > 0$, independent of $h$, such that for any initial guess $\Psi^0_{\rm dG}$ with $\vertiii{\Psi^0_{\rm dG}-\Psi_{\rm dG}}_{\rm dG} \leq \rho_1$, it follows 
	  $\vertiii{\Psi^n_{\rm dG}-\Psi_{\rm dG}}_{\rm dG} \leq \frac{\rho_1}{2^n}$ for all $n=1,2, \dots$ and  
  the iterates $\Psi^n_{\rm dG}$ of Newton's method are well-defined and converges quadratically to $\Psi_{\rm dG};$ that is, $
  \vertiii{\Psi_{\rm dG}^n-\Psi_{\rm dG}}_{\rm dG}
  \leq C_q \vertiii{\Psi_{\rm dG}^{n-1}-\Psi_{\rm dG}}_{\rm dG}^2 
  ,$ where $C_q$ is a constant independent of $h$.
	\end{thm}

	\section{Auxiliary results}\label{auxiliary results}
	This section presents some auxiliary results needed to establish the main results in Subsection \ref{section 3.1}. 
	 The boundedness and ellipticity results for $A_\dg(\cdot, \cdot)$, and the boundedness result for $B_\dg(\cdot,\cdot,\cdot, \cdot)$,  $C_\dg(\cdot, \cdot)$ are proved and 
	 the inf-sup conditions for a discrete linearized bilinear form and a perturbed bilinear form are established.
	\begin{lem}\emph{(Poincar\'e type inequality)}\cite{Poincare_Suli, Poincare_Brenner} \label{lem2.3.3}
		For $\phi \in H^1(\mathcal{T})$, there exists a constant $C_P>0$ independent of $h $ and $\phi$ such that for $1 \leq r < \infty$,	$\norm{\phi}_{L^r(\Omega)} \leq C_P \norm{\phi}_{\rm dG}.$
	\end{lem} 
	\medskip
	
	For boundedness and coercivity results of $A_\dg(\cdot,\cdot)$ and the boundedness results of $B_\dg(\cdot,\cdot,\cdot, \cdot)$ and $C_\dg(\cdot,\cdot)$, it is enough to prove the corresponding results for $a_\dg(\cdot,\cdot)$, $b_\dg(\cdot,\cdot, \cdot,\cdot)$ and $c_\dg(\cdot,\cdot)$, respectively.
	\begin{lem}\label{2.3.17}  \emph{(Boundedness and coercivity  of $ A_{\rm dG}(\cdot, \cdot)$)}
		\cite{Prudhomme}
		For	$\Theta_{\rm dG}, \Phi_{\rm dG} \in \V_h$,
		\begin{align*}
		A_{\rm dG}(\Theta_{\rm dG},\Phi_{\rm dG})\leq C_A \vertiii{\Theta_{\rm dG}}_{\rm dG} \vertiii{\Phi_{\rm dG}}_{\rm dG},
		\end{align*}
		where $C_A$ depends on the penalty parameter $\sigma$ and the constant from trace inequality.
		 For a sufficiently large parameter $\sigma$, there exists a 
		  positive constant $\alpha > 0$ such that 
		\begin{align*}
		\alpha \vertiii{\Phi_{\rm dG}}_{\rm dG}^2\leq A_{\rm dG}(\Phi_{\rm dG}, \Phi_{\rm dG}) \text{ for all } \Phi_{\rm dG} \in \V_h.
		\end{align*}
	\end{lem}
	\begin{lem}\label{2.3.31}\emph{(Boundedness of $B_\dg(\cdot,\cdot,\cdot, \cdot) $ and $ C_\dg(\cdot, \cdot)$)}
	For $\Xi,\boldsymbol \eta,\Theta,\Phi \in \mathbf{H}^1(\mathcal{T})$, 
	it holds that
	\begin{align}\label{eq:2.3.32}
	B_{\rm dG}(\Xi,\boldsymbol \eta,\Theta,\Phi ) \lesssim \epsilon^{-2} \vertiii{\Xi}_{\rm dG} \vertiii{\boldsymbol \eta}_{\rm dG} \vertiii{\Theta}_{\rm dG} \vertiii{\Phi}_{\rm dG} \text{ and, } C_{\rm dG}(\Theta,\Phi)\lesssim \epsilon^{-2} \norm{\Theta}_{\rm dG} \norm{\Phi}_{\rm dG} ,
	\end{align}
	and for all  $\Xi , \boldsymbol \eta \in \mathbf{H}^2(\Omega)$, $\Theta$, $\Phi \in  \mathbf{H}^1(\mathcal{T})$,
	\begin{align}\label{eq:2.3.33}
	B_{\rm dG}(\Xi,\boldsymbol \eta,\Theta,\Phi)\lesssim \epsilon^{-2}  \vertiii{\Xi}_2 \vertiii{\boldsymbol \eta}_2 \vertiii{\Theta}_0 \vertiii{\Phi}_0,
	\end{align}
	where the hidden constant in $"\lesssim"$ depends on the constants from $
	C_\epsilon$, $C_P$ and $C_S$.
\end{lem} 
	\begin{proof}
		We prove the boundedness results for $b_\dg(\cdot,\cdot,\cdot,\cdot)$ and $c_\dg(\cdot, \cdot)$. Then a use of definitions of $B_\dg(\cdot, \cdot,\cdot,\cdot)$ and $C_\dg(\cdot,\cdot)$, discrete Cauchy-Schwarz inequality and a grouping of the terms yields the required results. For $ \xi, \,  \eta,\, \theta$ and $\phi \in H^1(\mathcal{T})$, a use of Holder's inequality and Cauchy-Schwarz inequality along with Lemma \ref{lem2.3.3} leads to
		\begin{align} \label{4.3.1}
		b_\dg(\xi,   \eta, \theta, \phi ) \leq C_\epsilon \norm{\xi}_{0,4}\norm{ \eta}_{0,4}\norm{\theta}_{0,4}\norm{\phi}_{0,4}   \lesssim \epsilon^{-2} \norm{\xi}_\dg \norm{\eta}_\dg \norm{\theta}_\dg \norm{\phi}_\dg\\
		\text{ and, }
		c_\dg(\theta,\phi)
		\leq C_\epsilon \norm{\theta}_0 \norm{\phi}_0 \notag \lesssim \epsilon^{-2} \norm{\theta}_\dg \norm{\phi}_\dg \notag.
		\end{align}
		The proof of  \eqref{eq:2.3.33} follows analogously to that of  the proof of \eqref{eq:2.1.5.3} with a use of the 
		imbedding result $H^2(\Omega)\hookrightarrow L^\infty(\Omega)$ and Cauchy-Schwarz inequality.		
	\end{proof}
	\begin{rem}\label{4.5}
	A use of \eqref{4.3.1} leads to the following boundedness estimate 
	\begin{align*}
	\sum_{ T \in \mathcal{T}}\int_T (\Xi \cdot \boldsymbol \eta) (\Theta \cdot \Phi) \,{\rm dx} \lesssim \vertiii{\Xi}_{\rm dG}\vertiii{\boldsymbol \eta}_{\rm dG}\vertiii{\Theta}_{\rm dG}\vertiii{\Phi}_{\rm dG} \text{ for all } \Xi,\boldsymbol \eta,\Theta,\Phi \in \V_h.
	\end{align*}    		
    \end{rem}
	\noindent The next two lemmas describe the estimates for interpolation and enrichment operators that are crucial for the error estimates. 
	\begin{lem}\emph{(Interpolation estimate)}\cite{brenner} \label{2.3.35.1}
		For ${\rm v} \in H^s(\Omega) \text{ with } s \geq 1$, there exists  ${\rm{I}_{\rm dG}v} \in V_h$ such that for any $T \in \mathcal{T}$,
		\begin{align*}
		\norm{\rm v-\rm{I}_{\rm dG}v }_{H^l(T)} \leq C_I h_T ^{s-l} \norm{\rm v}_{H^s(T)},
		\end{align*}	
		for $l=0,1$ and some positive constant $C_I$ independent of $h$.
	\end{lem}
	\begin{lem} \emph{(Enrichment operator).} \cite{Poincare_Brenner,Karakashian} \label{2.3.35}
		Let ${\rm E}_h : V_h \rightarrow V_C\subset H^1(\Omega)$ be an enrichment operator with $V_C $ being the Lagrange $P_r$ conforming finite element space associated with the triangulation $\mathcal{T}$. Then for any $\phi_{\rm dG} \in V_h$, the following results hold:		
		\begin{align*}
		&\sum_{ T \in \mathcal{T}} h^{-2}\norm{{\rm E}_h\phi_{\rm dG} -\phi_{\rm dG}}_{L^2(T)}^2 + \norm{{\rm E}_h\phi_{\rm dG}}_{H^1(\Omega)}^2 \leq C_{en_1} \norm{\phi_{\rm dG}}_{\rm dG}^2, \\& \text{ and } \norm{{\rm E}_h\phi_{\rm dG} - \phi_{\rm dG}}_{\rm dG}^2 \leq C_{en_2}(\sum_{ E \in \mathcal{E}}\int_{ E}\frac{1}{h}[\phi_{\rm dG}]^2 \ds ) ,
		\end{align*}
		where $C_{en_1}$ and $C_{en_2}$ are constants independent of $h$.
	\end{lem}
	\noindent For all $\Theta_{\dg}, \Phi_{\dg} \in \V_h$,  define the discrete bilinear form $\dual{{\rm DN}_h(\Psi)\cdot, \cdot}$  by   $$ \dual{{\rm DN}_h(\Psi)\Theta_{\rm dG}, \Phi_{\rm dG}}: = A_{\rm dG}(\Theta_{\rm dG},\Phi_{\rm dG})+3B_{\rm dG}(\Psi,\Psi,\Theta_{\rm dG},\Phi_{\rm dG})+C_{\rm dG}(\Theta_{\rm dG},\Phi_{\rm dG}).$$ 
	\begin{thm}\label{2.4.1.1}
			Let $\Psi$ be a regular solution of the non-linear system
		\eqref{eq:2.1.5}. For a given fixed $\epsilon>0,$ a sufficiently large $\sigma$ and a sufficiently small discretization parameter chosen such that $h=O(\epsilon^{2} )$, there  exists a constant $\beta_0$ such that the following discrete inf-sup condition holds:
		\begin{align*}
		0< \beta_0 \leq \inf_{\substack{\Theta_{\rm dG} \in \V_h \\  \vertiii{\Theta_{\rm dG}}_{\rm dG}=1}} \sup_{\substack{\Phi_{\rm dG} \in \V_h \\  \vertiii{\Phi_{\rm dG}}_{\rm dG}=1}}\dual{{\rm DN}_h(\Psi)\Theta_{\rm dG}, \Phi_{\rm dG}} .
		\end{align*}	
	\end{thm}
	\begin{proof}
		For $\Theta \in \V + \V_h,$ Lemma \ref{2.3.31}, \eqref{eq:2.1.5.4} and \eqref{eq:2.1.5.3} yield $B_{\dg}(\Psi,\Psi,\Theta,\cdot) $,  $B(\Psi,\Psi,\Theta,\cdot), C_{\dg}(\Theta,\cdot)$ and $ C(\Theta,\cdot) \in \mathbf{L}^2(\Omega)$. Therefore, for a given $\Theta_{\dg} \in \V_h$ with $\vertiiidg{\Theta_{\dg}}=1$, there exist $\boldsymbol{\xi}_\epsilon$ and $\boldsymbol{ \eta}_\epsilon \in \mathbf{H}^2(\Omega)\cap \V $ that solve the linear systems 
		\begin{subequations}
			\begin{align}
			A(\boldsymbol \xi_\epsilon,\Phi)=3B_{\rm dG}(\Psi,\Psi,\Theta_{\rm dG},\Phi)+C_{\rm dG}(\Theta_{\rm dG},\Phi) \,\, \text{ for all } \Phi \in \V  \,\,\, \text{and}\label{2.4.2.1}\\
			A(\boldsymbol \eta_\epsilon,\Phi)=3B(\Psi,\Psi,{\rm E}_h\Theta_{\dg},\Phi)+C({\rm E}_h\Theta_{\dg},\Phi)  \,\, \text{ for all }  \Phi \in \V. \,\,\,\,\,\,\,\,\label{2.4.3.1}
			\end{align}	
		\end{subequations}	
		A use of Lemmas \ref{lem2.3.3}, \ref{2.3.31} and elliptic regularity leads to
		\begin{align}\label{2.4.4.0}
		\vertiii{\boldsymbol \xi_\epsilon}_2 \leq \vertiii{3B_{\rm dG}(\Psi,\Psi,\Theta_{\rm dG},\cdot)+C_{\rm dG}(\Theta_{\rm dG},\cdot) }_{\mathbf{L}^2}\lesssim \epsilon^{-2},
		\end{align}
		where the hidden constant in $"\lesssim"$ depends on $\vertiii{\Psi}_2$, $C_S$ and $C_P$.	
		Subtract \eqref{2.4.2.1} from \eqref{2.4.3.1}, choose $\Phi= \boldsymbol \eta_\epsilon-\boldsymbol \xi_\epsilon$ and use \eqref{2.3} and \eqref{eq:2.3.33} to obtain
		\begin{align}\label{2.4.4.1}
		C_{\alpha_0 }\vertiii{\boldsymbol \eta_\epsilon-\boldsymbol \xi_\epsilon}_1 \leq C_\epsilon( 3\vertiii{\Psi}_2^2+1)\vertiii{\Theta_{\dg}-{\rm E}_h\Theta_{\dg}}_0 , 
		\end{align}
		A use of Lemma \ref{2.3.35} in \eqref{2.4.4.1} yields 
		\begin{align}
		\vertiii{\boldsymbol \eta_\epsilon-\boldsymbol \xi_\epsilon}_1 \lesssim h\epsilon^{-2} , \label{2.4.4.2}
		\end{align}
		where the constant hidden in  $ "\lesssim" $ depends on $\vertiii{\Psi}_2, \alpha_0$ and $C_{en_1}$. 
		Since $\Psi$ is regular solution of \eqref{eq:2.1.5}, a use of \eqref{2.9} yields that there exists $\Phi \in \V$ with $\vertiii{\Phi}_1 =1$ such that
		\begin{align*}
		\beta \vertiii{\E_h \Theta_{\dg}}_1\leq & \dual{{\rm DN}(\Psi)\E_h \Theta_{\dg},\Phi}=A(\E_h \Theta_{\dg},\Phi)+3B(\Psi,\Psi,\E_h \Theta_{\dg},\Phi)+C(\E_h \Theta_{\dg},\Phi). 
		\end{align*}  
		A use of \eqref{2.4.3.1}, \eqref{2.3}, an introduction of intermediate terms 
		and triangle inequality leads to 
		\begin{align}
		\beta \vertiii{\E_h \Theta_{\dg}}_1 & \leq A(\E_h \Theta_{\dg} + \boldsymbol \eta_\epsilon,\Phi)\leq \vertiiidg{\E_h \Theta_{\dg} + \boldsymbol \eta_\epsilon}   \notag \\& \leq \vertiiidg{\E_h \Theta_{\dg} -\Theta_{\dg}}+ \vertiiidg{\Theta_{\dg} + \textrm{I}_\dg\boldsymbol \xi_\epsilon} +\vertiiidg{\textrm{I}_\dg\boldsymbol \xi_\epsilon -\boldsymbol \xi_\epsilon}     +  \vertiii{{\boldsymbol \xi_\epsilon}-{\boldsymbol \eta_\epsilon}}_1. \label{2.4.5.1}
		\end{align}	
		Since $\displaystyle [\boldsymbol \xi_\epsilon]=0$ on $\mathcal{E}_i$ and $\boldsymbol \xi_\epsilon=0$ on $\mathcal{E}_D$, a use of second inequality in Lemma \ref{2.3.35} and the triangle inequality yields
		\begin{align}\label{2.4.5.2}
		\vertiiidg{\E_h \Theta_{\dg} -\Theta_{\dg}} \leq C_{en_2} \vertiiidg{\Theta_{\dg}+\boldsymbol \xi_\epsilon} \leq C_{en_2} (\vertiiidg{\Theta_{\dg}+\textrm{I}_\dg\boldsymbol \xi_\epsilon}+ \vertiiidg{\boldsymbol \xi_\epsilon-\textrm{I}_\dg\boldsymbol \xi_\epsilon } ).
		\end{align} 
 		 Since $\Theta_{\dg} + \textrm{I}_\dg \boldsymbol \xi_\epsilon \in \V_h,$ Lemma \ref{2.3.17} implies that there exists $ \Phi^\epsilon_{\dg} \in \V_h$ with $\vertiiidg{\Phi^\epsilon_{\dg} }=1$ such that
		\begin{align}
		\alpha \vertiiidg{\Theta_{\dg} + \textrm{I}_\dg\boldsymbol \xi_\epsilon} &\leq A_\dg(\Theta_{\dg} + \textrm{I}_\dg\boldsymbol \xi_\epsilon,\Phi_{\dg}^\epsilon) \notag\\ & \leq \dual{{\rm DN}_h(\Psi)\Theta_{\dg},\Phi_{\dg}^\epsilon } + 3B_\dg(\Psi,\Psi,\Theta_{\dg},\E_h\Phi_{\dg}^\epsilon-\Phi_{\dg}^\epsilon)+ C_\dg(\Theta_{\dg},\E_h\Phi_{\dg}^\epsilon-\Phi_{\dg}^\epsilon) \notag\\& \quad+ A_\dg(\textrm{I}_\dg\boldsymbol \xi_\epsilon -\boldsymbol \xi_\epsilon,\Phi_{\dg}^\epsilon)+A_\dg(\boldsymbol \xi_\epsilon,\Phi_{\dg}^\epsilon-\E_h\Phi_{\dg}^\epsilon).\label{2.4.8.1} 
		\end{align}
		A use of  Lemmas \ref{2.3.31} and \ref{2.3.35} yields the estimates for the second and third terms in \eqref{2.4.8.1} as
		\begin{align}\label{4.7.1}
		3B_\dg(\Psi,\Psi,\Theta_{\dg},\E_h\Phi_{\dg}^\epsilon-\Phi_{\dg}^\epsilon)+ C_\dg(\Theta_{\dg},\E_h\Phi_{\dg}^\epsilon-\Phi_{\dg}^\epsilon) \lesssim  \epsilon^{-2}\vertiii{\E_h \Phi_{\dg}^\epsilon - \Phi_{\dg}^\epsilon}_0 \lesssim  h\epsilon^{-2}.
		\end{align}
		A use of Lemmas  \ref{2.3.17}, \ref{2.3.35.1} and \ref{2.3.35} leads to
		\begin{align}\label{4.7.2}
		A_\dg(\textrm{I}_\dg\boldsymbol \xi_\epsilon -\boldsymbol \xi_\epsilon,\Phi_{\dg}^\epsilon)+A_\dg(\boldsymbol \xi_\epsilon,\Phi_{\dg}^\epsilon-\E_h\Phi_{\dg}^\epsilon) 
		\lesssim  {h \vertiii{\boldsymbol \xi_\epsilon}_2} +  \vertiii{\boldsymbol \xi_\epsilon}_2\vertiii{\E_h \Phi_{\dg}^\epsilon- \Phi_{\dg}^\epsilon}_0
		\lesssim h \vertiii{\boldsymbol \xi_\epsilon}_2 .
		\end{align}
		A combination of \eqref{4.7.1} and \eqref{4.7.2} in \eqref{2.4.8.1} leads to
		\begin{align}\label{4.9}
		\alpha \vertiiidg{\Theta_{\dg} + \textrm{I}_\dg\boldsymbol \xi_\epsilon}
		\lesssim \dual{{\rm DN}_h(\Psi)\Theta_{\dg},\Phi_{\dg}^\epsilon }+ h \vertiii{\boldsymbol \xi_\epsilon}_2 +  h\epsilon^{-2},
		\end{align}
		where $"\lesssim"$ includes a constant that depends on $\vertiii{\Psi}_2, C_A, C_S,C_P, C_{en_1} $ and $C_I.$ A substitution of \eqref{4.9} in \eqref{2.4.5.2} and a use of Lemma \ref{2.3.35.1} yields
		\begin{align}\label{4.14}
			\vertiiidg{\E_h \Theta_{\dg} -\Theta_{\dg}} \lesssim \dual{{\rm DN}_h(\Psi)\Theta_{\dg},\Phi_{\dg}^\epsilon }+ h \vertiii{\boldsymbol \xi_\epsilon}_2 +  h\epsilon^{-2} .
		\end{align}	 			
		Moreover, a use of \eqref{4.9}, \eqref{4.14}, Lemma \ref{2.3.35.1} and \eqref{2.4.4.2} in \eqref{2.4.5.1} yields
		\begin{align}\label{4.15}
		\vertiii{\E_h \Theta_{\dg}}_1 \lesssim  \dual{{\rm DN}_h(\Psi)\Theta_{\dg},\Phi_{\dg}^\epsilon }+ h \vertiii{\boldsymbol \xi_\epsilon}_2 +  h\epsilon^{-2} .
		\end{align}
		A use of triangle inequality, \eqref{4.14}, \eqref{4.15} and \eqref{2.4.4.0} leads to
		\begin{align} \label{2.4.7.1}
		1=\vertiiidg{\Theta_{\dg}} \leq \vertiiidg{ \Theta_{\dg}-\E_h \Theta_{\dg}} + \vertiii{\E_h \Theta_{\dg}}_1  \leq C_1(\dual{{\rm DN}_h(\Psi)\Theta_{\dg},\Phi_{\dg}^\epsilon }+ h\epsilon^{-2}),  
		\end{align}	
		where the constant $C_1$ is independent of $h$ and $\epsilon$.
		Therefore, for a given $\epsilon$, the discrete inf-sup condition holds with $\beta_0= \frac{1}{2C_1}$ for $h < h_0:=\frac{\epsilon^2}{2C_1}  $.
	\end{proof}
	We use the perturbed bilinear form defined as   $\dual{{\rm DN}_h(\textrm{I}_\dg\Psi)\Theta_{\rm dG}, \Phi_{\rm dG}} := A_{\rm dG}(\Theta_{\rm dG},\Phi_{\rm dG})+3B_{\rm dG}(\textrm{I}_\dg\Psi,\textrm{I}_\dg\Psi,\Theta_{\rm dG},\Phi_{\rm dG})+C_{\rm dG}(\Theta_{\rm dG},\Phi_{\rm dG}) $ for all $\Theta_{\dg}, \Phi_{\dg} \in \V_h$ in our analysis and  Newton's algorithm in our manuscript.   
	The next lemma establishes discrete inf-sup condition for the perturbed bilinear form.
	\begin{lem} \label{2.4.12} \emph{(Stability of perturbed bilinear form).}
		Let $\Psi$ be a regular solution of
		\eqref{eq:2.1.5} and ${{\rm I}_{\rm dG}}\Psi$ be the interpolation of  $\Psi$ from Lemma \ref{2.3.35.1}. For a given fixed $\epsilon>0,$ a sufficiently large $\sigma$ and a sufficiently small discretization parameter chosen as $h=O(\epsilon^{2} )$, the perturbed bilinear form $\dual{{\rm DN}_h({\rm I}_{\rm dG}\Psi)\cdot, \cdot}$
	    satisfies the following discrete inf-sup condition 
		\begin{align*}
		0< \frac{\beta_0}{2} \leq \inf_{\substack{\Theta_{\rm dG} \in \V_h \\  \vertiii{\Theta_{\rm dG}}_{\rm dG}=1}} \sup_{\substack{\Phi_{\rm dG} \in \V_h \\  \vertiii{\Phi_{\rm dG}}_{\rm dG}=1}}\dual{{\rm DN}_h({{\rm I}_{\rm dG}\Psi})\Theta_{\rm dG},\Phi_{\rm dG} }.
		\end{align*}
	\end{lem}
	\begin{proof}
		For $\tilde{\Psi}=\Psi- {\rm I}_{\dg}\Psi$, 
		\begin{align*}
		\dual{{\rm DN}_h({{\rm I}_{\rm dG}\Psi})\Theta_{\rm dG},\Phi_{\rm dG} } 	= A_\dg(\Theta_{\dg},\Phi_{\dg})+3B_\dg(\Psi- \tilde{\Psi},\Psi- \tilde{\Psi},\Theta_{\dg},\Phi_{\dg})+C_\dg(\Theta_{\dg},\Phi_{\dg}). 
		\end{align*}
		A use of \eqref{3.2} and Remark \ref{4.5} lead to
		\begin{align*} 
		&B_\dg(\Psi- \tilde{\Psi},\Psi- \tilde{\Psi},\Theta_{\dg},\Phi_{\dg})=  B_\dg(\Psi,\Psi,\Theta_{\dg},\Phi_{\dg}) +B_\dg(\tilde{\Psi}, \tilde{\Psi},\Theta_{\dg},\Phi_{\dg})\\ &\quad - \frac{4 }{3}\epsilon^{-2} \sum_{ T \in \mathcal{T}}\int_T ((\Psi \cdot \tilde{\Psi})(\Theta_{\dg} \cdot \Phi_{\dg})+(\tilde{\Psi}\cdot \Theta_{\dg})(\Psi \cdot \Phi_{\dg})+(\Psi \cdot \Theta_{\dg})(\tilde{\Psi} \cdot \Phi_{\dg})) \dx \\
		&\geq  B_\dg(\Psi,\Psi,\Theta_{\dg},\Phi_{\dg}) +B_\dg(\tilde{\Psi}, \tilde{\Psi},\Theta_{\dg},\Phi_{\dg}) - 4C_P\epsilon^{-2} \tilde{\vertiii{\Psi}}_\dg \vertiiidg{\Psi} \vertiiidg{\Theta_{\dg}} \vertiiidg{\Phi_{\dg}}.
		\end{align*}
		\begin{align*}
		\text{Therefore}, \, \dual{{\rm DN}_h({{\rm I}_{\rm dG}\Psi})\Theta_{\rm dG},\Phi_{\rm dG} } \geq & \dual{{\rm DN}_h(\Psi)\Theta_{\rm dG},\Phi_{\rm dG} } + 3B_\dg( \tilde{\Psi},\tilde{\Psi},\Theta_{\dg},\Phi_{\dg})\quad\quad\quad\quad\quad\quad\quad\quad\quad\quad\quad \\&- 12C_PC_\epsilon \tilde{\vertiii{\Psi}}_\dg \vertiiidg{\Psi} \vertiiidg{\Theta_{\dg}} \vertiiidg{\Phi_{\dg}}. 
		\end{align*}
		A use of Theorem \ref{2.4.1.1}, \eqref{eq:2.3.32} and Lemma \ref{2.3.35.1} leads to
		\begin{align*}
		\sup_{\vertiiidg{\Phi_{\dg}}=1}\dual{{\rm DN}_h({{\rm I}_{\rm dG}\Psi})\Theta_{\rm dG},\Phi_{\rm dG} } \geq & \sup_{\vertiiidg{\Phi_{\dg}}=1}\dual{{\rm DN}_h(\Psi)\Theta_{\rm dG},\Phi_{\rm dG} }  - C_2(h  + 1  ) h \epsilon^{-2} \vertiiidg{\Theta_{\dg}}\\
		\geq & (\beta_0 -  C_2h\epsilon^{-2})\vertiiidg{\Theta_{\dg}},
		\end{align*}
		where the constant $C_2$ depends on $\vertiii{\Psi}_2$, $C_P$, $C_S$, and $ C_I$.
	Hence, for a sufficiently small choice of $h< h_2:=\min(h_0,h_1)$ with $h_1:=\frac{\beta_0}{2C_2}\epsilon^2$ , the required result holds for $\frac{\beta_0}{2}.$ 
	\end{proof}
	\section{Proof of main results} \label{main result}
	The proofs of main results use some preliminary results that are established first.   
		First  the non-linear map $\mu^\epsilon_{\dg}$ is defined on the discrete space $\V_h$ and in Theorem \ref{thm2.5.1} it is established that $\mu^\epsilon_{\dg}$ maps a closed convex set to itself and the fixed point of $\mu^\epsilon_{\dg}$ is a solution of the discrete non-linear problem in \eqref{2.3.13} and vice-versa. 
	For $\Phi_{\dg} \in \V_h$, define the map, $\mu^\epsilon_{\dg}:\V_h \rightarrow \V_h$ by
	\begin{align}\label{2.5.1}
	\dual{{\rm DN}_h({{\rm I}_{\rm dG}\Psi}) \mu^\epsilon_{\dg}(\Theta_{\dg}),\Phi_{\rm dG} }
	= 3B_\dg(\textrm{I}_{\dg}\Psi, \textrm{I}_{\dg}\Psi,\Theta_{\dg},\Phi_{\dg}) - B_\dg(\Theta_{\dg},\Theta_{\dg}, \Theta_{\dg},\Phi_{\dg}) +L_\dg(\Phi_{\dg}).
	\end{align}
	The map $\mu^\epsilon_{\dg}$ is well-defined and this follows from the inf-sup conditions of  $\dual{{\rm DN}_h({{\rm I}_{\rm dG}\Psi}) \cdot,\cdot }$ in Lemma \ref{2.4.12}. We use the abbreviation $\mu_{\dg}:=\mu^\epsilon_{\dg}$ in the rest of the article for notational convenience.
	Let $\mathbb{B}_R(\textrm{I}_{\dg}\Psi):= \{\Phi_{\dg} \in \V_h: \vertiiidg{\textrm{I}_{\dg}\Psi-\Phi_{\dg}} \leq R\}.$
	The existence and uniqueness result of the discrete solution $\Psi_{\dg}$ in Theorem  \ref{2.5.2} is an application of Brouwer's fixed point theorem. Then, the energy norm error estimate is presented. As an alternative approach, the existence and uniqueness of discrete solution is established using Newton-Kantorovich theorem and is followed by the best approximation result. 
	\begin{thm}[Mapping of ball to ball] \label{thm2.5.1}
		Let $\Psi$ be a regular solution of the non-linear system
		\eqref{eq:2.1.5}. For a given fixed $\epsilon>0,$ a sufficiently large $\sigma$ and a sufficiently small discretization parameter chosen as $h=O(\epsilon^{2+\alpha} )$ with $\alpha > 0$, there exists a positive constant $R(h)$ such that $\mu_{\rm dG}$ maps the ball $\mathbb{B}_{R(h)}({\rm I}_{\rm dG}\Psi)$ to itself; 
		\begin{align*}
		\vertiii{\Theta_{\rm dG} - {\rm I}_{\rm dG}\Psi }_{\rm dG} \leq R(h) \implies \vertiii{\mu_{\rm dG}(\Theta_{\rm dG}) - {\rm I}_{\rm dG}\Psi}_{\rm dG} \leq R(h) \text{ for all } \Theta_{\rm dG} \in \V_{h} .	
		\end{align*}
	\end{thm}
	\begin{proof}
		The solution of \eqref{eq:2.1.5} that belongs to $\Psi \in \h^2(\Omega)\cap \V$ satisfies 
		the dG formulation
		\begin{align}\label{2.5.3}
		A_{\dg}(\Psi,\Phi_{\dg})+B_{\dg}(\Psi,\Psi,\Psi,\Phi_{\dg})+C_{\dg}(\Psi,\Phi_{\dg})=L_\dg(\Phi_{\dg}) \,\,\,\,\text{ for all } \Phi_{\dg} \in \V_h.
		\end{align}
		A use of the definition and linearity of $\dual{{\rm DN}_h({{\rm I}_{\rm dG}\Psi})\cdot, \cdot }$, \eqref{2.5.1} and \eqref{2.5.3} leads to
		\begin{align}\label{2.5.4}
		&\dual{{\rm DN}_h({{\rm I}_{\rm dG}\Psi}) (\textrm{I}_{\dg}\Psi -\mu_{\rm dG}(\Theta_{\dg})),\Phi_{\rm dG} } =\dual{{\rm DN}_h({{\rm I}_{\rm dG}\Psi}) \textrm{I}_{\dg}\Psi ,\Phi_{\rm dG} } -\dual{{\rm DN}_h({{\rm I}_{\rm dG}\Psi}) \mu_{\rm dG}(\Theta_{\dg}),\Phi_{\rm dG} } \notag\\&
		=A_{\dg}(\textrm{I}_{\dg}\Psi , \Phi_{\dg} ) + 3B_{\dg}(\textrm{I}_{\dg}\Psi,\textrm{I}_{\dg}\Psi,\textrm{I}_{\dg}\Psi , \Phi_{\dg} )+ C_{\dg}(\textrm{I}_{\dg}\Psi , \Phi_{\dg} ) -3B_{\dg}(\textrm{I}_{\dg}\Psi,\textrm{I}_{\dg}\Psi, \Theta_{\dg},\Phi_{\dg})\notag\\ & \quad+ B_{\dg}(\Theta_{\dg},\Theta_{\dg}, \Theta_{\dg},\Phi_{\dg}) -L_{\rm dG}(\Phi_{\dg})\notag\\&
		=A_{\dg}(\textrm{I}_{\dg}\Psi -\Psi , \Phi_{\dg} )+C_{\dg}(\textrm{I}_{\dg}\Psi -\Psi , \Phi_{\dg} ) +(B_{\dg}(\textrm{I}_{\dg}\Psi,\textrm{I}_{\dg}\Psi,\textrm{I}_{\dg}\Psi , \Phi_{\dg} ) -B_{\dg}(\Psi,\Psi, \Psi,\Phi_{\dg}))\notag\\
		& \quad + (2B_{\dg}(\textrm{I}_{\dg}\Psi,\textrm{I}_{\dg}\Psi,\textrm{I}_{\dg}\Psi , \Phi_{\dg} ) -3B_{\dg}(\textrm{I}_{\dg}\Psi,\textrm{I}_{\dg}\Psi, \Theta_{\dg},\Phi_{\dg}) + B_{\dg}(\Theta_{\dg},\Theta_{\dg}, \Theta_{\dg},\Phi_{\dg}))\notag\\&=: T_1+T_2+T_3+T_4. 
		\end{align}
		A use of Lemmas \ref{2.3.17}, \ref{2.3.31} and \ref{2.3.35.1} yields
		\begin{align*}
		&T_1:=A_{\dg}(\textrm{I}_{\dg}\Psi -\Psi , \Phi_{\dg} )  \lesssim h \vertiiidg{\Phi_{\dg} }.
		\\&
		T_2:=C_{\dg}(\textrm{I}_{\dg}\Psi -\Psi , \Phi_{\dg} ) \lesssim \epsilon^{-2} \vertiii{\textrm{I}_{\dg}\Psi -\Psi }_0 \vertiiidg{\Phi_{\dg} } \lesssim \epsilon^{-2} h^2  \vertiiidg{\Phi_{\dg} }.
		\end{align*}
		A rearrangement of the terms in $B_{\rm dG}(\cdot,\cdot,\cdot,\cdot)$ and a use Lemmas \ref{2.3.31} and \ref{2.3.35.1} lead to
		\begin{align*}
		&T_3:=B_{\dg}(\textrm{I}_{\dg}\Psi,\textrm{I}_{\dg}\Psi,\textrm{I}_{\dg}\Psi , \Phi_{\dg} ) -B_{\dg}(\Psi,\Psi, \Psi,\Phi_{\dg})\notag\\&=B_{\dg}(\textrm{I}_{\dg}\Psi-\Psi,\textrm{I}_{\dg}\Psi-\Psi, \textrm{I}_{\dg}\Psi, \Phi_{\dg} )+2B_{\dg}(\textrm{I}_{\dg}\Psi-\Psi,\textrm{I}_{\dg}\Psi-\Psi, \Psi,\Phi_{\dg}) +3B_{\dg}( \Psi,\Psi,\textrm{I}_{\dg}\Psi-\Psi,\Phi_{\dg})
		\\ &\lesssim  \epsilon^{-2} (\vertiiidg{\textrm{I}_{\dg}\Psi-\Psi}^2\vertiiidg{\textrm{I}_{\dg}\Psi}+\vertiiidg{\textrm{I}_{\dg}\Psi-\Psi}^2\vertiii{\Psi}_2+\vertiii{\Psi}_2^2\vertiii{\textrm{I}_{\dg}\Psi-\Psi}_0)\vertiiidg{\Phi_{\dg}}\lesssim \epsilon^{-2} h^2 \vertiiidg{\Phi_{\dg}}.
		\end{align*}
		 Set $\tilde{\mathbf{e}}=\Theta_{\dg}-\textrm{I}_{\dg}\Psi$ and use definition of $ B_{\dg}(\cdot,\cdot, \cdot,\cdot)$, a regrouping of terms and Remark \ref{4.5} to obtain
		\begin{align}
		T_4: &=2B_{\dg}(\textrm{I}_{\dg}\Psi,\textrm{I}_{\dg}\Psi,\textrm{I}_{\dg}\Psi , \Phi_{\dg} ) -3B_{\dg}(\textrm{I}_{\dg}\Psi,\textrm{I}_{\dg}\Psi, \Theta_{\dg},\Phi_{\dg}) + B_{\dg}(\Theta_{\dg},\Theta_{\dg}, \Theta_{\dg},\Phi_{\dg}) \notag
		\\& 
		=C_\epsilon \sum_{ T \in \mathcal{T}}\int_T(\abs{\Theta_{\dg}}^2-\abs{\textrm{I}_{\dg}\Psi}^2)(\Theta_{\dg} \cdot\Phi_{\dg})\dx+ 2 C_\epsilon \sum_{ T \in \mathcal{T}}\int_T (\textrm{I}_{\dg}\Psi - \Theta_{\dg})\cdot\textrm{I}_{\dg}\Psi(\textrm{I}_{\dg}\Psi \cdot \Phi_{\dg})\dx \notag\\& =   C_\epsilon \sum_{ T \in \mathcal{T}}\int_T
		\tilde{\mathbf{e}}\cdot(\tilde{\mathbf{e}}+2\textrm{I}_{\dg}\Psi)(\tilde{\mathbf{e}}+\textrm{I}_{\dg}\Psi) \cdot\Phi_{\dg}\dx-2C_\epsilon \sum_{ T \in \mathcal{T}}\int_T (\tilde{\mathbf{e}}\cdot\textrm{I}_{\dg}\Psi) (\textrm{I}_{\dg}\Psi\cdot \Phi_{\dg})\dx         \notag    \\&
		 =C_\epsilon \sum_{ T \in \mathcal{T}}\int_T (\tilde{\mathbf{e}}\cdot\tilde{\mathbf{e}}) (\tilde{\mathbf{e}}\cdot \Phi_{\dg})\dx
		+ 2C_\epsilon\sum_{ T \in \mathcal{T}}\int_T (\tilde{\mathbf{e}}\cdot \textrm{I}_{\dg}\Psi) (\tilde{\mathbf{e}}\cdot \Phi_{\dg})\dx+C_\epsilon\sum_{ T \in \mathcal{T}}\int_T (\tilde{\mathbf{e}}\cdot\tilde{\mathbf{e}}) (\textrm{I}_{\dg}\Psi\cdot \Phi_{\dg})\dx\notag\\
		&
		\lesssim \epsilon^{-2}\vertiiidg{\tilde{\mathbf{e}}}^2(\vertiiidg{\tilde{\mathbf{e}}}+\vertiii{\textrm{I}_{\dg}\Psi}_{\dg})\vertiiidg{\Phi_{\dg}}\notag.
		\end{align}
		The hidden constants in $"\lesssim"$ in the estimates of $T_1, T_2$, $T_3$ and $T_4$ depend on  $\vertiii{\Psi}_2$, $C_A,$ $C_P$, $C_S$, and $C_I$.
		A substitution of the estimates for $T_1$, $T_2$, $T_3$ and $T_4$ in \eqref{2.5.4} with a use of $\vertiii{{\rm I}_{\rm dG}\Psi}_{\rm dG}\lesssim \vertiii{\Psi}_2$ leads to
		\begin{align}\label{2.5.12}
		\dual{{\rm DN}_h({{\rm I}_{\rm dG}\Psi}) (\textrm{I}_{\dg}\Psi -\mu_{\rm dG}(\Theta_{\dg})),\Phi_{\rm dG} } \lesssim \big((h+\epsilon^{-2} h^2)+ \epsilon^{-2}\vertiiidg{\tilde{\mathbf{e}}}^2(\vertiiidg{\tilde{\mathbf{e}}}+1)\big) \vertiiidg{\Phi_{\dg}}.
		\end{align}
		A use of Lemma \ref{2.4.12} yields that there exists 
		a $\Phi_{\dg} \in \V_h$ with $\vertiiidg{\Phi_{\dg} }=1$ such that
		\begin{align}\label{2.5.13}
		\frac{\beta_0}{2}\vertiiidg{\textrm{I}_{\dg}\Psi -\mu_{\rm dG}(\Theta_{\dg})} \leq
		\dual{{\rm DN}_h({{\rm I}_{\rm dG}\Psi}) (\textrm{I}_{\dg}\Psi -\mu_{\rm dG}(\Theta_{\dg})),\Phi_{\rm dG} }. 
		\end{align}
		A use of \eqref{2.5.12} and 
		$\vertiiidg{\tilde{\mathbf{e}}}\leq R(h)$ in \eqref{2.5.13} leads to 
		\begin{align}\label{4.22}
		\vertiiidg{\textrm{I}_{\dg}\Psi -\mu_{\rm dG}(\Theta_{\dg})}\lesssim (h+\epsilon^{-2} h^2)+ \epsilon^{-2}\vertiiidg{\tilde{\mathbf{e}}}^2(\vertiiidg{\tilde{\mathbf{e}}}+1) \leq C_3(h+\epsilon^{-2} h^2+\epsilon^{-2}
		{R(h)}^2(R(h)+1)),
		\end{align} 
		where $C_3$ is a positive constant that depends on $\vertiii{\Psi}_2$, $\beta_0$, $C_A$, $C_S$, $C_P$ and $C_I$. 
		Assume $h\leq \epsilon^{2+\alpha}$ with $\alpha>0$ so that $ h \epsilon^{-2}\leq  h^{\frac{\alpha}{2+\alpha}}$. Choose $R(h)=2C_3h.$ 
		 For $h< h_{4}:=\min({h_2,h_3})
		$ with $h_{3}^{\frac{\alpha}{2+\alpha}}< \frac{1}{2(1+4C^2_3)^2}< \frac{1}{2}$, we obtain
		\begin{align*}
		\vertiiidg{\textrm{I}_{\dg}\Psi -\mu_{\rm dG}(\Theta_{\dg})}\leq & C_3h+C_3hh^{\frac{\alpha}{2+\alpha}}+ 4 C_3^3hh^{\frac{\alpha}{2+\alpha}}(2C_3h+1)  \\ \leq&   C_3h\left(1+h^{\frac{\alpha}{2+\alpha}}(1+4C_3^2)+8C_3^3hh^{\frac{\alpha}{2+\alpha}} \right)\leq { C_3h\left(1+\frac{1}{2}+\frac{1}{2}\frac{8C_3^3h}{(1+4C_3^2)^2}\right)} .
		\end{align*}
		Since $h<h_{3}< \frac{1}{2^{\frac{2+\alpha}{\alpha}}}<1$ and $\frac{8C_3^3}{(1+4C_3^2)^2}<1,$ $\vertiiidg{\textrm{I}_{\dg}\Psi -\mu_{\rm dG}(\Theta_{\dg})}\leq 2C_3h=R(h).$
		This completes the proof of Theorem \ref{thm2.5.1} .
	\end{proof}
\begin{rem}
	We derive error estimates with  $h$-$\epsilon$ dependency,  for a given fixed $\epsilon$. This provides a sufficient condition on the choice of the discretization parameter $h$ for a given fixed $\epsilon$ that ensures convergence. A large value of $\alpha$ would require a very small value of $h$. Equally, a very small  choice of $\alpha$ would require $h \to 0$ from the estimates above. In this respect, the choices $\alpha \ll 1$ and $\alpha \gg1$ lead to computationally expensive scenarios and we only focus on bounded $O(1)$ values of $\alpha$ in this manuscript.
\end{rem}
\begin{lem}[Contraction result]\label{4.3}
	 For a given fixed $\epsilon>0,$ a sufficiently large $\sigma$ and a sufficiently small discretization parameter chosen as $h=O(\epsilon^{2+\alpha} )$ with $\alpha > 0$, 
 the following contraction result holds: $\text{for } \Theta_1,  \Theta_2 \in \mathbb{B}_{R(h)}({\rm I}_{\rm dG}\Psi)$,
	\begin{align*} 
	\vertiii{\mu_{\rm dG}(\Theta_1)-\mu_{\rm dG}(\Theta_2)}_{\rm dG}\leq Ch^{\frac{\alpha}{2+\alpha}}\vertiii{\Theta_1-\Theta_2}_{\rm dG} 
	\end{align*}
	 with some positive constant $C$ independent of $h$ and $\epsilon$ .
\end{lem}	
\begin{proof}
		Let $\Theta_1$ and  $\Theta_2 \in \mathbb{B}_{R(h)}(\textrm{I}_{\dg}\Psi)$.  
		For $\Phi_{\dg} \in \V_h$, a use of  \eqref{2.5.1}, the definition and linearity of $\dual{{\rm DN}_h({{\rm I}_{\rm dG}\Psi}) \cdot, \cdot }$, an elementary manipulation and grouping of term lead to 
		\begin{align*}
		&\dual{{\rm DN}_h({{\rm I}_{\rm dG}\Psi}) (\mu_{\rm dG}(\Theta_1)- \mu_{\rm dG}(\Theta_2)),\Phi_{\rm dG} } \notag\\&=3B_\dg(\textrm{I}_{\dg}\Psi,\textrm{I}_{\dg}\Psi, \Theta_1,\Phi_{\dg}) - B_\dg(\Theta_1,\Theta_1, \Theta_1,\Phi_{\dg})- 3B_\dg(\textrm{I}_{\dg}\Psi,\textrm{I}_{\dg}\Psi, \Theta_2,\Phi_{\dg}) + B_\dg(\Theta_2,\Theta_2, \Theta_2,\Phi_{\dg})\notag\\&
		=C_\epsilon \sum_{ T \in \mathcal{T}}\int_T \big(2\textrm{I}_{\dg}\Psi \cdot (\Theta_1-\Theta_2)(\textrm{I}_{\dg}\Psi \cdot \Phi_{\dg})+ (\abs{\textrm{I}_{\dg}\Psi}^2-\abs{\Theta_1}^2)(\Theta_1 \cdot\Phi_{\dg}) - (\abs{\textrm{I}_{\dg}\Psi}^2-\abs{\Theta_2}^2)(\Theta_2 \cdot\Phi_{\dg}) \big)\dx \notag\\&
		=2C_\epsilon \sum_{ T \in \mathcal{T}}\int_T \big(\textrm{I}_{\dg}\Psi \cdot (\Theta_1-\Theta_2)(\textrm{I}_{\dg}\Psi-\Theta_2) \cdot \Phi_{\dg} +  \textrm{I}_{\dg}\Psi \cdot (\Theta_1-\Theta_2)(\Theta_2 \cdot \Phi_{\dg})		\big)\dx
		 \notag\\& \quad +C_\epsilon \sum_{ T \in \mathcal{T}}\int_T \big((\textrm{I}_{\dg}\Psi -\Theta_1 )\cdot(\Theta_1-\textrm{I}_{\dg}\Psi )(\Theta_1-\Theta_2) \cdot\Phi_{\dg} +   2(\textrm{I}_{\dg}\Psi -\Theta_1 )\cdot \textrm{I}_{\dg}\Psi(\Theta_1-\Theta_2) \cdot\Phi_{\dg}				
		\notag \\& \quad+ (\textrm{I}_{\dg}\Psi -\Theta_1 )\cdot(\textrm{I}_{\dg}\Psi+\Theta_1 )(\Theta_2 \cdot\Phi_{\dg})\big)\dx
		-C_\epsilon \sum_{ T \in \mathcal{T}}\int_T (\textrm{I}_{\dg}\Psi -\Theta_2 )\cdot(\textrm{I}_{\dg}\Psi +\Theta_2 )(\Theta_2 \cdot\Phi_{\dg})\dx \notag\\& 
		=2C_\epsilon \sum_{ T \in \mathcal{T}}\int_T \textrm{I}_{\dg}\Psi \cdot (\Theta_1-\Theta_2)(\textrm{I}_{\dg}\Psi-\Theta_2) \cdot \Phi_{\dg}\dx
				+C_\epsilon\sum_{ T \in \mathcal{T}}\int_T(\textrm{I}_{\dg}\Psi -\Theta_1 )\cdot(\Theta_1-\textrm{I}_{\dg}\Psi )(\Theta_1-\Theta_2) \cdot\Phi_{\dg}\dx\notag 
				 \\& \quad
				+2C_\epsilon \sum_{ T \in \mathcal{T}}\int_T (\textrm{I}_{\dg}\Psi -\Theta_1 )\cdot \textrm{I}_{\dg}\Psi(\Theta_1-\Theta_2) \cdot\Phi_{\dg}\dx \notag
				\\& 	\quad	
				+C_\epsilon\sum_{ T \in \mathcal{T}}\int_T(\Theta_1-\Theta_2)\cdot ((\textrm{I}_{\dg}\Psi-\Theta_1)+(\textrm{I}_{\dg}\Psi-\Theta_2))((\Theta_2-\textrm{I}_{\dg}\Psi)+\textrm{I}_{\dg}\Psi) \cdot \Phi_{\dg}\dx
		\end{align*}
		Set $\tilde{\mathbf{e}}_1=\textrm{I}_{\dg}\Psi -\Theta_1$, $\tilde{\mathbf{e}}_2=\textrm{I}_{\dg}\Psi -\Theta_2$ and $\mathbf{e}=\Theta_1 -\Theta_2$. 	A use of Remark \ref{4.5}, $\vertiii{\tilde{\e}_1}_\dg\leq 2C_3h$ and $\vertiii{\tilde{\e}_2}_\dg\leq 2C_3h$  yields
		\begin{align}\label{2.5.17}
 	&\dual{{\rm DN}_h({{\rm I}_{\rm dG}\Psi}) (\mu_{\rm dG}(\Theta_1)- \mu_{\rm dG}(\Theta_2)),\Phi_{\rm dG} } \notag\\& =2C_\epsilon\sum_{ T \in \mathcal{T}}\int_T  (\textrm{I}_{\dg}\Psi \cdot \mathbf{e})(\tilde{\mathbf{e}}_2  \cdot \Phi_{\dg})\dx  +C_\epsilon\sum_{ T \in \mathcal{T}}\int_T(-\tilde{\mathbf{e}}_1 \cdot \tilde{\mathbf{e}}_1 )(\mathbf{e} \cdot\Phi_{\dg})\dx +  2C_\epsilon\sum_{ T \in \mathcal{T}}\int_T(\tilde{\mathbf{e}}_1 \cdot \textrm{I}_{\dg}\Psi)(\mathbf{e} \cdot\Phi_{\dg})\dx 
		\notag\\&
		\quad+ C_\epsilon \sum_{ T \in \mathcal{T}}\int_T \mathbf{e} \cdot (\tilde{\mathbf{e}}_1+\tilde{\mathbf{e}}_2) (-\tilde{\mathbf{e}}_2 +\textrm{I}_{\dg}\Psi) \cdot \Phi_{\dg}\dx \notag\\&
		\lesssim \epsilon^{-2}  (\vertiii{\tilde{\mathbf{e}}_2}_\dg+\vertiii{\tilde{\mathbf{e}}_1}_\dg+\vertiii{\tilde{\mathbf{e}}_1}^2_\dg+\vertiii{\tilde{\mathbf{e}}_2}^2_\dg)\vertiii{\mathbf{e}}_\dg \vertiii{\Phi_{\dg}}_\dg   \lesssim \epsilon^{-2}h(h+1)\vertiii{\e}_\dg\vertiii{\Phi_{\dg}}_\dg.
		\end{align}
		A use of \eqref{2.5.17} and Lemma \ref{2.4.12} yields that there exists a $\Phi_{\dg} \in \V_h$ with $\vertiii{\Phi_{\dg}}_\dg=1$ such that 
		\begin{align}\label{4.24}
		 \vertiii{\mu_{\rm dG}(\Theta_1)- \mu_{\rm dG}(\Theta_2)}_\dg \leq \frac{2}{\beta_0}\dual{{\rm DN}_h({{\rm I}_{\rm dG}\Psi}) (\mu_{\rm dG}(\Theta_1)- \mu_{\rm dG}(\Theta_2)),\Phi_{\rm dG} } \lesssim \epsilon^{-2} h(h+1)\vertiii{\e}_\dg.
		\end{align}
		The hidden constant in $"\lesssim"$ depends on $\vertiii{\Psi}_2$, $\beta_0$,  $C_P$ and $C_3$.
		A use of $h\leq \epsilon^{2+\alpha} $ with $\alpha>0$ in \eqref{4.24} implies $h\epsilon^{-2}\leq h^{\frac{\alpha}{2+\alpha}}$ and this completes the proof of Lemma \ref{4.3}. 
	\end{proof}
\begin{rem}
	Note that we require only $h=O(\epsilon^2)$ to prove the discrete inf-sup conditions in Theorem \ref{2.4.1.1} and Lemma \ref{2.4.12}, whereas we need $h=O(\epsilon^{2+\alpha})$ with $\alpha>0$ to prove Theorem \ref{thm2.5.1} and Lemma \ref{4.3}.
\end{rem}
Now we present the proof of results stated in Subsection \ref{section 3.1}.
	\begin{proof}[\textbf{Proof of Theorem  \ref{2.5.2}}]
		Let $\Psi$ be a regular solution of the non-linear system
		\eqref{eq:2.1.5} and let $\Psi_{\rm{dG} }$ solve \eqref{2.3.13}. 
		A use of Theorem \ref{thm2.5.1} yields that $\mu_{\rm dG}$ maps a non-empty convex closed subset $\mathbb{B}_{R(h)}(\textrm{I}_{\dg}\Psi)$ of a finite dimensional vector space $\V_h$ to itself. Also, $\mu_{\rm dG}$ is continuous. Therefore an application of the Brouwer fixed point theorem \cite{KesavaTopicsFunctinal} yields that $\mu_{\rm dG}$ has at least one fixed point, say $\Psi_{\dg}$ in this ball $\mathbb{B}_{R(h)}(\textrm{I}_{\dg}\Psi)$ (for details see \ref{Brouwer}). That is,
		\begin{align}\label{5.9}
		\vertiiidg{\Psi_{\dg }-\textrm{I}_{\dg}\Psi}\leq 2C_3h.
		\end{align}
		 The contraction result in Lemma \ref{4.3} establishes the uniqueness of the solution of \eqref{2.3.13} for a sufficiently small $h$. 
		The proof of error estimate is straightforward using Lemma \ref{2.3.35.1} and \eqref{5.9}.
	\end{proof}
	An alternative approach using Newton-Kantorovich theorem also provides an explicit formula for the radius $R(h)$ of the ball $\mathbb{B}_{R(h)}(\textrm{I}_{\dg}\Psi)$ and proves the existence and uniqueness of discrete solution.
	
	\noindent The Newton scheme below is motivated by $\Psi_{\rm dG }^n$ and $\Psi_{\rm dG }^{n-1}$ in place of $\Psi_{\rm dG }$ and  $\textrm{I}_{\dg}\Psi$, respectively, a substitution of $\mu_{\dg}(\Psi_{\dg}^{n})=\Psi_{\dg}^{n}$ in \eqref{2.5.1} and a use of definition of $\dual{{\rm DN}_h({{\rm I}_{\rm dG}\Psi})\cdot,\cdot }$ in Lemma \ref{2.4.12}. These substitutions yield     
	\begin{align}\label{eq:3.1.1}
	&A_\dg( \Psi_{\dg}^{n},\Phi_{\dg})+3B_\dg(\Psi_{\dg}^{n-1},\Psi_{\dg}^{n-1}, \Psi_{\dg}^{n},\Phi_{\dg})+C_\dg( \Psi_{\dg}^{n},\Phi_{\dg})\notag\\&= 2B_\dg(\Psi_{\dg}^{n-1},\Psi_{\dg}^{n-1}, \Psi_{\dg}^{n-1},\Phi_{\dg})+L_{\rm dG}(\Phi_{\rm dG}), \text{ where } n=1,2, \dots.
	\end{align}
\begin{thm}\cite{Kantorovich_1948,zeidler, keller}\label{thm:4.9}
	Suppose that the mapping ${\rm N}_h : D \subset \V_h \rightarrow \V_h$ is Fr\'echet differentiable on a open convex set $D$, and the derivative  
	${\rm DN}_h(\cdot)$ is Lipschitz continuous on $D$ with Lipschitz constant $L$.
	For a fixed starting point $x^0_{\rm dG} \in D, $ the inverse ${\rm DN}_h(x^0_{\rm dG})^{-1}$ exists as a continuous operator on $\V_h.$ The real numbers $a$ and $b$ are chosen such that 
	\begin{align}\label{eq:4.9}
	\vertiii{{\rm DN}_h(x^0_{\rm dG})^{-1}}_{L(\V^*_h;\V_h )} \leq a \quad \text{and} \quad \vertiii{{\rm DN}_h(x^0_{\rm dG})^{-1}N_h(x^0_{\rm dG})}_{\rm dG } \leq b
	\end{align}
	and $h*:=abL\leq \frac{1}{2}.$ Also, the first approximation $x_{\rm dG}^1:= {x^0_{\rm dG}}- {\rm DN}_h(x^0_{\rm dG})^{-1}{\rm N}_h(x^0_{\rm dG}) $ has a property that the closed ball $\bar{U}(x_{\rm dG }^1;r)$ lies within the domain of definition $D,$  where $r= \frac{1-\sqrt{1-2h^*}}{aL}-b.$
	Then the following are true.
	\begin{enumerate}
		\item {Existence and uniqueness}. There exists a solution $x_{\rm dG }\in \bar{U}(x_{\rm dG }^1;r)$ and the solution is unique on $\bar{U}(x^0_{\rm dG};r*)\cap D,$ that is on a suitable neighborhood of the initial point ${x^0_{\rm dG}}$ with $r*=\frac{1+\sqrt{1-2h^*}}{aL}. $ 
		\item Convergence of Newton's method. The Newton's scheme with initial iterate $x^0_{\rm dG}$ leads to a sequence $x^n_{\rm dG}$ in $\bar{U}(x^0_{\rm dG};r*)$, which converges to $x_{\rm dG }$ with error bound 
		\begin{align}\label{5.12}
		\vertiii{x^n_{\rm dG}-x_{\rm dG}}_{\rm dG }\leq \frac{(1-(1-2h^*)^{\frac{1}{2}})^{2^n}}{{2^n}aL}, \quad n=0,1 \dots.
		\end{align}
	\end{enumerate}
	\begin{thm}[Existence and uniqueness of discrete solution]
		Let $\Psi$ be a regular solution of the continuous non-linear system ${\rm N}(\Psi; \Phi)=0$ for all $\Phi \in \V.$	For a given fixed $\epsilon>0,$ a sufficiently large $\sigma$ and a sufficiently small discretization parameter chosen as $h=O(\epsilon^{2+\alpha} )$ with $\alpha > 0$, the following results hold true:
		\begin{enumerate}
			\item there exists a solution $\Psi_{\rm dG } \in \V_h$ to ${\rm N}_h(\Psi_{\rm dG }; \Phi_{\rm dG} )=0 $ for all $\Phi_{\rm dG} \in \V_h$ such that $\vertiii{\Psi -\Psi_{\rm dG } }_{\rm dG }\leq \rho,$ where $\rho= C_\rho(h+b+r)$ with $b, r $ defined in Theorem \ref{thm:4.9}. The constant $C_\rho$ depends on $\vertiii{\Psi}_2$ and $C_I$,  
			\item there is at most one solution $\Psi_{\rm dG }$ to ${\rm N}_h(\Psi_{\rm dG }; \Phi_{\rm dG} )=0 $ for all $\Phi_{\rm dG} \in \V_h$ in $\bar{U}({\rm I}_{\rm dG}\Psi;r*)\cap D$ with $r* =\frac{1+\sqrt{1-C_K h\epsilon^{-2}(1+h\epsilon^{-2})}}{2\beta_0^{-1}C_L\epsilon^{-2}},$ where the constant $C_K$ depends on $\vertiii{\Psi}_2$, $ C_A$, $ C_S$, $ C_P$, $ C_I$, $\beta_0$ and $C_L $, a constant from Lipschitz continuity of ${\rm DN}_h(\cdot),$
			\item the sequence $\Psi^n_{\rm dG }$ of iterates converges to $\Psi_{\rm dG }$ and the error bound is given by
			\begin{align}\label{4.26.1}
			\vertiii{\Psi^n_{\rm dG }-\Psi_{\rm dG }}_{\rm dG }\leq \frac{(1-(1-C_K h\epsilon^{-2}(1+h\epsilon^{-2}))^{\frac{1}{2}})^{2^n}}{{2^n}(2\beta_0^{-1}C_L\epsilon^{-2})}, \quad n=0,1 \dots.
			\end{align}
		\end{enumerate}
	\end{thm}
	\begin{proof}
		A use of the  definition of ${\rm DN}_h(\cdot)$, Lemma \ref{2.3.31} and a simple manipulation leads to the fact that ${\rm DN}_h(\cdot)$ is Lipschitz continuous on $D$ with Lipschitz constant $L=C_L\epsilon^{-2},$ where $C_L$ is a constant independent of $\epsilon.$ For a choice of 
	 $x^0_\dg={\rm I}_{\rm dG}\Psi$, Lemma \ref{2.4.12} yields 
		\begin{align} \label{4.9.1}
		\vertiii{{\rm DN}_h({{\rm I}_{\rm dG}\Psi})^{-1}}_{L(\V^*_h;\V_h )} \leq a
		\end{align}
		with $a=2\beta_0^{-1}.$
		Given $\Phi_{\dg} \in \V_\dg$ with $\vertiiidg{\Phi_{\dg} }=1,$  \eqref{2.5.3} leads to
		\begin{align*}
		&{\rm N}_h({{\rm I}_{\rm dG}\Psi};\Phi_{\dg})=A_{\dg}({{\rm I}_{\rm dG}\Psi},\Phi_{\dg})+B_{\dg}({{\rm I}_{\rm dG}\Psi},{{\rm I}_{\rm dG}\Psi},\Psi_{\dg},\Phi_{\dg})+C_{\dg}({{\rm I}_{\rm dG}\Psi},\Phi_{\dg})-L_\dg(\Phi_{\dg}) 
		\\	&= A_{\dg}(\textrm{I}_{\dg}\Psi -\Psi , \Phi_{\dg} )+C_{\dg}(\textrm{I}_{\dg}\Psi -\Psi , \Phi_{\dg} ) +(B_{\dg}(\textrm{I}_{\dg}\Psi,\textrm{I}_{\dg}\Psi,\textrm{I}_{\dg}\Psi , \Phi_{\dg} ) -B_{\dg}(\Psi,\Psi, \Psi,\Phi_{\dg}))\\&= T_1+ T_2+ T_3.
		\end{align*}
		A use of the estimates of $T_1, T_2, T_3$ from \eqref{2.5.4} yields 
		\begin{align} \label{4.9.2}
		{\rm N}_h({{\rm I}_{\rm dG}\Psi};\Phi_{\dg})\leq C_N(h+h^2\epsilon^{-2}),
		\end{align}
		where the constant $C_N$ depends on $\vertiii{\Psi}_2$, $C_A$, $ C_S$, $ C_P$ and $ C_I$.
		A use of \eqref{4.9.1} and \eqref{4.9.2} yields 
		\begin{align*}
		\vertiii{{\rm DN}_h({\rm I}_{\rm dG}\Psi)^{-1}{\rm N}_h({\rm I}_{\rm dG}\Psi)}_{\rm dG } \leq \vertiii{{\rm DN}_h({{\rm I}_{\rm dG}\Psi})^{-1}}_{L(\V^*_h;\V_h )} \vertiii{{\rm N}_h({{\rm I}_{\rm dG}\Psi})}_{\V^*_h }\leq b,
		\end{align*}
		where $b=2C_N\beta_0^{-1} h(1+h\epsilon^{-2}).$ 
		  A sufficiently small choice of $h\leq \epsilon^{2+\alpha}$ with $\alpha>0,$ leads to  $h^*= abL=4C_NC_L\beta_0^{-2} h\epsilon^{-2}(1+h\epsilon^{-2})\leq 4C_NC_L\beta_0^{-2} h^{\frac{\alpha}{2+\alpha}} (1+h^{\frac{\alpha}{2+\alpha}} )< C_Kh^{\frac{\alpha}{2+\alpha}} $ with $C_K=8C_NC_L\beta_0^{-2}$. For a choice of  $h< h_6:=\min(h_4, h_5)$ with $h_5^{\frac{\alpha}{2+\alpha}}<\frac{1}{2C_K},$ $h^*<\frac{1}{2}.$  Therefore, an application of Theorem \ref{thm:4.9} yields the existence of the discrete solution $\Psi_{\rm dG }$ with 
		$r= \frac{1-\sqrt{1-C_K h\epsilon^{-2}(1+h\epsilon^{-2})}}{2\beta_0^{-1}C_L \epsilon^{-2}}-2C_N\beta_0^{-1} h(1+h\epsilon^{-2})$ and $r*=\frac{1+\sqrt{1-C_K h\epsilon^{-2}(1+h\epsilon^{-2})}}{2\beta_0^{-1}C_L\epsilon^{-2}}$. A use of the second part of Theorem \ref{thm:4.9} yields the error bound $\vertiii{{\rm I}_{\rm dG}\Psi - \Psi_{\rm dG }}_{\rm dG}  \leq \frac{1-\sqrt{1-C_K h\epsilon^{-2}(1+h\epsilon^{-2})}}{2\beta_0^{-1}C_L\epsilon^{-2}}=O(h) $ which justifies the choice of $R(h)$ in Theorem \ref{thm2.5.1}.	
		A use of triangle inequality, Lemma \ref{2.3.35.1} and the second estimate in \eqref{eq:4.9} for the first Newton correction leads to
		\begin{align*}
		\vertiii{\Psi - \Psi_{\rm dG }}_{\rm dG} \leq \vertiii{\Psi - {\rm
				I}_{\rm dG}\Psi}_{\rm dG}+\vertiii{{\rm
				I}_{\rm dG}\Psi - \Psi^1_{\rm dG }}_{\rm dG}+\vertiii{\Psi^1_{\rm dG } - \Psi_{\rm dG }}_{\rm dG} \leq C_\rho (h+ b+ r)=:\rho.
		\end{align*} 
		A substitution of $h^*, a $ and $L$ in \eqref{5.12} yields the error bound in the Newton's convergence given by \eqref{4.26.1}.
	\end{proof}
\end{thm}
\begin{proof}[\textbf{Proof of Theorem  \ref{thmbestapproximation}}]
	Let $\Psi^*_{\dg }$ be the best approximation of $\Psi$ in $\V_h$. Then  $$\vertiiidg{\Psi-\Psi^*_{\dg }}=\min_{\Theta_{\rm dG} \in \V_h}\vertiii{\Psi- \Theta_{\rm dG}}_{\rm dG}.$$ Set $\mathbf{e}_\dg= \Psi^*_{\rm dG }-\Psi_{\rm dG } \in \V_h. $ A use of Theorem \ref{2.4.1.1} yields that there exists a $\Phi_{\rm dG} \in \V_h$ with $\vertiiidg{\Phi_{\rm dG}} = 1 $ such that 
\begin{align}\label{4.29}
\beta_0\vertiiidg{\mathbf{e}_\dg}\leq \dual{{\rm DN}_h(\Psi)\mathbf{e}_\dg, \Phi_{\rm dG}}. 
\end{align}
Set $\mathbf{e}_1= \Psi-\Psi_{\rm dG } $. A use of Taylor series expansion of ${\rm N}_h(\cdot; \cdot)$ around $\Psi$ leads to 
\begin{align*}
{\rm N}_h(\Psi_{\rm dG }; \Phi_{\rm dG})= {\rm N}_h(\Psi; \Phi_{\rm dG}) - \dual{{\rm DN}_h(\Psi)\mathbf{e}_1, \Phi_{\rm dG}} +\frac{1}{2}\dual{{\rm D}^2{\rm N}_h(\Psi)(\mathbf{e}_1)\mathbf{e}_1, \Phi_{\rm dG}}-\frac{1}{6}\dual{{\rm D}^3{\rm N}_h(\Psi)(\mathbf{e}_1)(\mathbf{e}_1)\mathbf{e}_1, \Phi_{\rm dG}}.
\end{align*}
A use of \eqref{2.3.13} and \eqref{2.5.3} lead to
\begin{align*}
0=  \dual{{\rm DN}_h(\Psi)\mathbf{e}_1, \Phi_{\rm dG}} -\frac{1}{2}\dual{{\rm D}^2{\rm N}_h(\Psi)(\mathbf{e}_1)\mathbf{e}_1, \Phi_{\rm dG}}+\frac{1}{6}\dual{{\rm D}^3{\rm N}_h(\Psi)(\mathbf{e}_1)(\mathbf{e}_1)\mathbf{e}_1, \Phi_{\rm dG}}.
\end{align*}
Rewrite $\mathbf{e}_1= (\Psi-\Psi^*_{\rm dG })+ \mathbf{e}_\dg$. A use of the linearity of $\dual{{\rm DN}_h(\Psi)\cdot, \cdot}$, \eqref{4.29} and the above equality leads to
\begin{align}\label{4.30}
\dual{{\rm DN}_h(\Psi)\mathbf{e}_\dg, \Phi_{\rm dG}}=\dual{{\rm DN}_h(\Psi)(\Psi^*_{\rm dG }-\Psi), \Phi_{\rm dG}}+\frac{1}{2}\dual{{\rm D}^2{\rm N}_h(\Psi)(\mathbf{e}_1)\mathbf{e}_1, \Phi_{\rm dG}}-\frac{1}{6}\dual{{\rm D}^3{\rm N}_h(\Psi)(\mathbf{e}_1)(\mathbf{e}_1)\mathbf{e}_1, \Phi_{\rm dG}}.
\end{align}
Since $\dual{{\rm DN}_h(\Psi)\mathbf{e}_1, \Phi_{\rm dG}}=A_{\rm dG}(\mathbf{e}_1,\Phi_{\rm dG})+3B_{\rm dG}(\Psi,\Psi,\mathbf{e}_1,\Phi_{\rm dG})+C_{\rm dG}(\mathbf{e}_1,\Phi_{\rm dG}),$ $\dual{{\rm D}^2{\rm N}_h(\Psi)(\mathbf{e}_1)\mathbf{e}_1, \Phi_{\rm dG}} = 6 B_\dg(\mathbf{e}_1,\mathbf{e}_1, \Psi, \Phi_{\rm dG})$ and $\dual{{\rm D}^3{\rm N}_h(\Psi)(\mathbf{e}_1)(\mathbf{e}_1)\mathbf{e}_1, \Phi_{\rm dG}}=6B_\dg(\mathbf{e}_1,\mathbf{e}_1, \mathbf{e}_1, \Phi_{\rm dG}),$ a use of Lemmas \ref{2.3.17} and \ref{2.3.31} leads to
\begin{align}\label{4.31}
\vertiii{{\rm DN}_h(\Psi)}_{\mathbf{L}^2}\lesssim (1+\epsilon^{-2}),\, \vertiii{{\rm D}^2{\rm N}_h(\Psi)}_{\mathbf{L}^2}\lesssim \epsilon^{-2} \text{ and } \vertiii{{\rm D}^3{\rm N}_h(\Psi)}_{\mathbf{L}^2}\lesssim \epsilon^{-2}.
\end{align}
The constant in $"\lesssim"$ depends on $\vertiii{\Psi}_2$,  $C_A$, $C_S$ and $C_P$.
A use of \eqref{4.29} and \eqref{4.31} in \eqref{4.30} yields
\begin{align}\label{4.32}
\beta_0 \vertiiidg{\mathbf{e}_\dg}\lesssim (1+\epsilon^{-2})\vertiii{\Psi^*_{\rm dG }-\Psi}_{\rm dG}+\epsilon^{-2} \vertiii{\mathbf{e}_1}_{\rm dG}^2+\epsilon^{-2} \vertiii{\mathbf{e}_1}_{\rm dG}^3.
\end{align}
The triangle inequality and \eqref{4.32} leads to 
\begin{align}\label{5.19}
\vertiiidg{\mathbf{e}_1}\leq \vertiii{\Psi- \Psi^*_{\rm dG }}_{\rm dG}+ \vertiii{\mathbf{e}_\dg}_{\rm dG} \lesssim (1+\epsilon^{-2})\vertiii{\Psi^*_{\rm dG }-\Psi}_{\rm dG}+\epsilon^{-2}(\vertiii{\mathbf{e}_1}_{\rm dG}^2+\vertiii{\mathbf{e}_1}_{\rm dG}^3). 
\end{align}
For a sufficiently small choice of the discretization parameter $h=O(\epsilon^{2+\alpha}) $ with $\alpha>0$, use  $\vertiii{\mathbf{e}_1}_{\rm dG}\leq C_e h$ in Theorem \ref{2.5.2} and \eqref{5.19} to obtain 
\begin{align*}
(1+\epsilon^{-2})\vertiii{\Psi^*_{\rm dG }-\Psi}_{\rm dG} \geq (C_4-\epsilon^{-2}(\vertiii{\mathbf{e}_1}_{\rm dG}+\vertiii{\mathbf{e}_1}_{\rm dG}^2))\vertiiidg{\mathbf{e}_1} \geq (C_4-h\epsilon^{-2}(C_e+C_e^2))\vertiiidg{\mathbf{e}_1}, 
\end{align*}
where the constant $C_4$ depends on  $\vertiii{\Psi}_2$, $C_A$, $C_S$, $C_P$ and $\beta_0$.
Since $h\epsilon^{-2}\leq h^{\frac{\alpha}{2+\alpha}}$, a sufficiently small choice of $h < h_8:=\min(h_6, h_7)$ with $h_7^{\frac{\alpha}{2+\alpha}}=\frac{C_4}{2(C_e+C_e^2)}$ completes the proof of Theorem \ref{thmbestapproximation} for $C_B=\frac{2}{C_4}.$	 
\end{proof}
 The next two lemmas are required to prove Theorem \ref{sipg}. 
\begin{lem} \label{5.6}
	For $\lambda = -1$, any $\Psi \in \mathbf{H}^2(\Omega)$,  $\boldsymbol{\chi} \in \mathbf{H}^2(\Omega) \cap \V$ and the interpolation ${\rm I}_{\rm dG}\Psi \in \V_h$ satisfy
	\begin{align*}
	A_{\rm dG}({\rm I}_{\rm dG}\Psi -\Psi, \boldsymbol{\chi}) \lesssim h^2\vertiii{\Psi}_2\vertiii{\boldsymbol{\chi}}_2.
	\end{align*}
\end{lem}	
\begin{proof}
	A use of definition of $A_{\rm dG}(\cdot, \cdot)$, $[\boldsymbol{\chi}]=0 $ on $\mathcal{E}_i,$ $\boldsymbol{\chi} = 0$ on $\mathcal{E}_D$ and integration by parts yields
	\begin{align*}
	&A_{\rm dG}({\rm I}_{\rm dG}\Psi -\Psi, \boldsymbol{\chi})= \sum_{ T \in \mathcal{T}} \int_T \nabla ({\rm I}_{\rm dG}\Psi -\Psi) \cdot \nabla \boldsymbol{\chi} \dx -  \sum_{E \in \mathcal{E}} \int_{E} \{\frac{\partial \boldsymbol{\chi}}{\partial \eta}\} \cdot [{\rm I}_{\rm dG}\Psi -\Psi] \ds\\& = -\sum_{ T \in \mathcal{T}} \int_T ({\rm I}_{\rm dG}\Psi -\Psi) \cdot \Delta \boldsymbol{\chi} \dx + \sum_{T \in \mathcal{T}} \int_{\partial T} \frac{\partial \boldsymbol{\chi}}{\partial \eta} \cdot ({\rm I}_{\rm dG}\Psi -\Psi) \ds	-  \sum_{E \in \mathcal{E}} \int_{E} \{\frac{\partial \boldsymbol{\chi}}{\partial \eta}\} \cdot [{\rm I}_{\rm dG}\Psi -\Psi] \ds.	
	\end{align*}
	We have $\sum_{T \in \mathcal{T}} \int_{\partial T} \frac{\partial \boldsymbol{\chi}}{\partial \eta} \cdot ({\rm I}_{\rm dG}\Psi -\Psi) \ds= \sum_{E \in \mathcal{E}_i} \int_{E} [\frac{\partial \boldsymbol{\chi}}{\partial \eta}] \cdot \{{\rm I}_{\rm dG}\Psi -\Psi\} \ds+\sum_{E \in \mathcal{E}} \int_{E}\{\frac{\partial \boldsymbol{\chi}}{\partial \eta}\} \cdot [{\rm I}_{\rm dG}\Psi -\Psi] \ds.$
	Since $\boldsymbol{\chi} \in \mathbf{H}^2(\Omega),$ $[\nabla {\boldsymbol{\chi}}]=0$ for all $E \in \mathcal{E}_i.$ A use of Cauchy-Schwartz inequality and  Lemma \ref{2.3.35.1} leads to, 
	\begin{align*}
	A_{\rm dG}({\rm I}_{\rm dG}\Psi -\Psi, \boldsymbol{\chi})= -\sum_{ T \in \mathcal{T}} \int_T ({\rm I}_{\rm dG}\Psi -\Psi) \cdot \Delta \boldsymbol{\chi} \dx \leq \vertiii{{\rm I}_{\rm dG}\Psi -\Psi}_0 \vertiii{\boldsymbol{\chi}}_2 \lesssim h^2 \vertiii{\Psi}_2 \vertiii{\boldsymbol{\chi}}_2,
	\end{align*}
	where the constant suppressed in $"\lesssim"$ depends on $C_I.$ This concludes the proof.
\end{proof}
\noindent For given $G \in \mathbf{L}^2(\Omega)$ and $\Psi$ that solves \eqref{2.3.1.2}, consider the linear dual problem: 
\begin{align}\label{4.9.1.1}
-\Delta \boldsymbol{\chi} + \frac{2}{\epsilon^2}(\abs{\Psi}^2 \boldsymbol{\chi} + 2 (\Psi \cdot \boldsymbol{\chi})\Psi)-\frac{2}{\epsilon^2} \boldsymbol{\chi}= G \,\, \text{in} \,\,\Omega  \,\, \text{ and } \boldsymbol{\chi} =0 \text{ on } \,\, \partial \Omega.   
\end{align}
The weak formulation that corresponds to \eqref{4.9.1.1} seeks $\boldsymbol{\chi} \in \V$ such that
\begin{align}\label{dualweak}
\dual{{\rm DN}(\Psi)\Phi, \boldsymbol{\chi}}:= A(\Phi, \boldsymbol{\chi})+ B(\Psi,\Psi,\Phi, \boldsymbol{\chi})+ C(\Phi, \boldsymbol{\chi})=(G, \Phi) \text{ for all }  \Phi \in \V.
\end{align}
A use of \eqref{2.9} establishes the well-posedness of \eqref{dualweak}. 
\begin{lem}\label{xiregularity}
	A solution $\boldsymbol{\chi} $ of \eqref{dualweak} belongs to $  \mathbf{H}^2(\Omega) \cap \V$  and satisfies $\vertiii{\boldsymbol{\chi}}_2 \lesssim (1+\epsilon^{-2})\vertiii{G}_0,$
	where the hidden constant in $"\lesssim"$ depends on $\vertiii{\Psi}_2, C_S $ and $\beta.$  	
\end{lem} 
\begin{proof}
	A use of \eqref{2.9} and \eqref{dualweak} yields 
	\begin{align}\label{5.22}
	\beta \vertiii{\boldsymbol{\chi}}_1 \leq  \sup_{\substack{\Phi \in \V \\  \vertiii{\Phi}_1=1}}\dual{{\rm DN}(\Psi)\Phi, \boldsymbol{\chi}} = \sup_{\substack{\Phi \in \V \\  \vertiii{\Phi}_1=1}}(G, \Phi) \leq \vertiii{G}_0. 
	\end{align} 
	A use of  \eqref{eq:2.1.5.4} and \eqref{eq:2.1.5.3} implies that $ \vertiii{-3B(\Psi,\Psi, \cdot, \boldsymbol{\chi}) -C(\cdot, \boldsymbol{\chi})+ G }_0 \lesssim \epsilon^{-2}\vertiii{\boldsymbol{\chi}}_1 + \vertiii{G}_0 $. 
	Hence, the elliptic regularity \cite{Evance19} with a boot-strapping argument and \eqref{5.22} completes the proof.
\end{proof}
\begin{proof}[\textbf{Proof of Theorem  \ref{sipg}}]
	Set $G=\textrm{I}_{\dg}\Psi -\Psi_\dg $ in \eqref{4.9.1.1}.  Multiply \eqref{4.9.1.1} by $\Phi_{\rm dG}=\textrm{I}_{\dg}\Psi -\Psi_\dg$ and integrate by parts to obtain
	\begin{align} \label{4.12}
	\dual{{\rm DN}_h(\Psi)\textrm{I}_{\dg}\Psi -\Psi_\dg, \boldsymbol{\chi}}=
	\vertiii{\textrm{I}_{\dg}\Psi -\Psi_\dg}_0^2,
	\end{align}
	Here, $\dual{{\rm DN}_h(\Psi)\textrm{I}_{\dg}\Psi -\Psi_\dg, \boldsymbol{\chi}}=  A_{\rm dG}(\textrm{I}_{\dg}\Psi -\Psi_\dg, \boldsymbol{\chi}) + B_{\rm dG}(\Psi, \Psi, \textrm{I}_{\dg}\Psi -\Psi_\dg, \boldsymbol{\chi}) + C_{\rm dG}(\textrm{I}_{\dg}\Psi -\Psi_\dg, \boldsymbol{\chi}).$ Note that the terms that involve $[\boldsymbol{\chi}]$ in the definition of $A_{\rm dG}(\cdot, \cdot)$ are zero. However, they are retained for the ease of further manipulations.
	 A use of  \eqref{2.3.13} and the fact that $\Psi \in \mathbf{H}^2(\Omega)$ satisfies the discrete formulation  \eqref{2.3.13} leads to
	\begin{align}
	&\vertiii{ \textrm{I}_{\dg}\Psi -\Psi_\dg}_0^2  =\dual{{\rm DN}_h(\Psi)\textrm{I}_{\dg}\Psi -\Psi_\dg, \boldsymbol{\chi}} 
	=\dual{{\rm DN}_h(\Psi)\textrm{I}_{\dg}\Psi -\Psi_\dg, \boldsymbol{\chi}}  + {\rm N}_\dg(\Psi_\dg, \textrm{I}_{\dg}\boldsymbol{\chi}) -{\rm N}_\dg(\Psi,\textrm{I}_{\dg} \boldsymbol{\chi})\notag \\&
	=  (A_\dg(\textrm{I}_{\dg}\Psi -\Psi, \boldsymbol{\chi}) 
	+A_\dg(\Psi - \Psi_\dg ,\boldsymbol{\chi} -\textrm{I}_{\dg} \boldsymbol{\chi})) +(C_\dg(\textrm{I}_{\dg}\Psi -\Psi, \boldsymbol{\chi})   
	+C_\dg(\Psi - \Psi_\dg ,\boldsymbol{\chi} -\textrm{I}_{\dg} \boldsymbol{\chi})) \notag\\&\quad
	+ (3B_\dg(\Psi,\Psi, \textrm{I}_{\dg}\Psi- \Psi_\dg, \boldsymbol{\chi}- \textrm{I}_{\dg}\boldsymbol{\chi} )  + 3 B_\dg(\Psi ,\Psi , \textrm{I}_{\dg}\Psi -\Psi, \textrm{I}_{\dg}\boldsymbol{\chi}))+ (B_\dg(\Psi_\dg,\Psi_\dg, \Psi_\dg , \textrm{I}_{\dg}\boldsymbol{\chi})  \notag\\& \quad  
	-3B_\dg(\Psi, \Psi, \Psi_\dg, \textrm{I}_{\dg}\boldsymbol{\chi})  +2B_\dg(\Psi,\Psi, \Psi, \textrm{I}_{\dg}\boldsymbol{\chi})) \notag\\&:= T_5+ T_6+T_7+T_8. \label{4.10}
	\end{align}
	Lemmas \ref{2.3.17}, \ref{2.3.31}, \ref{2.3.35.1} and \ref{5.6} and Theorem \ref{2.5.2}  yield 
	\begin{align*}
	T_5:= & A_\dg(\textrm{I}_{\dg}\Psi -\Psi, \boldsymbol{\chi}) 
	+A_\dg(\Psi - \Psi_\dg ,\boldsymbol{\chi} -\textrm{I}_{\dg} \boldsymbol{\chi})  
	\lesssim  h^2  \vertiii{\boldsymbol{\chi}}_2. \notag \\
	T_6:= & C_\dg(\textrm{I}_{\dg}\Psi -\Psi, \boldsymbol{\chi})  
	+C_\dg(\Psi - \Psi_\dg ,\boldsymbol{\chi} -\textrm{I}_{\dg} \boldsymbol{\chi})  
	\lesssim  \epsilon^{-2} h^2  \vertiii{\boldsymbol{\chi}}_2.
	\end{align*}
	Lemmas \ref{2.3.31} and \ref{2.3.35.1}  lead to
	\begin{align*}
	T_7&:=  3B_\dg(\Psi,\Psi, \textrm{I}_{\dg}\Psi- \Psi_\dg, \boldsymbol{\chi}- \textrm{I}_{\dg}\boldsymbol{\chi} ) +3 B_\dg(\Psi , \Psi , \textrm{I}_{\dg}\Psi -\Psi, \textrm{I}_{\dg}\boldsymbol{\chi}) 
	\\& \lesssim  \epsilon^{-2} ( \vertiii{\Psi}_2^2 \vertiii{ \textrm{I}_{\dg}\Psi- \Psi_\dg}_\dg \vertiii{\boldsymbol{\chi}- \textrm{I}_{\dg}\boldsymbol{\chi}}_\dg +  \vertiii{\Psi}_2^2 \vertiii{ \textrm{I}_{\dg}\Psi- \Psi}_0 \vertiii{\textrm{I}_{\dg}\boldsymbol{\chi}}_0 )
	\lesssim \epsilon^{-2} h^2 \vertiii{\boldsymbol{\chi}}_2 .
	\end{align*}
	Set $\mathbf{e}_3=  \Psi_{\dg }-\Psi $ and estimate  $T_8$ as in $T_4$ of Theorem \ref{thm2.5.1} and use Theorem \ref{2.5.2} to obtain
	\begin{align*}
	T_8& :=2B_\dg(\Psi, \Psi, \Psi, \textrm{I}_{\dg}\boldsymbol{\chi}) + B_\dg(\Psi_\dg,\Psi_\dg, \Psi_\dg , \textrm{I}_{\dg}\boldsymbol{\chi}) -3B_\dg(\Psi,\Psi, \Psi_\dg, \textrm{I}_{\dg}\boldsymbol{\chi}) \\& \lesssim \epsilon^{-2} \vertiiidg{\mathbf{e}_3 }^2(\vertiiidg{\mathbf{e}_3 }+\vertiiidg{\Psi })\vertiii{\boldsymbol{\chi}}_2 \lesssim\epsilon^{-2} h^2 (h+1) \vertiii{\boldsymbol{\chi}}_2.
	\end{align*}
	A combination of the estimates for $T_5$, $T_6$, $T_7$ and $T_8$ in  \eqref{4.10} and a use of Lemma \ref{xiregularity} yields
	\begin{align}\label{5.5}
	\vertiii{ \textrm{I}_{\dg}\Psi -\Psi_\dg}_0  \lesssim h^2 (1+\epsilon^{-2} )^2.
	\end{align} 
	A use of \eqref{5.5} and triangle inequality yields
	\begin{align*}
	\vertiii{ \Psi -\Psi_\dg}_0 \leq \vertiii{\Psi- \textrm{I}_{\dg}\Psi }_0+\vertiii{ \textrm{I}_{\dg}\Psi -\Psi_\dg}_0  \lesssim h^2(1+ (1+\epsilon^{-2} )^2),
	\end{align*} 
	where the constants suppressed in $"\lesssim"$ depends on  $\vertiii{\Psi}_2$, $\beta_0$, $C_A$, $C_S$, $C_P$ and $C_I$. 
	This completes the proof of Theorem  \ref{sipg}.
\end{proof}	
	\section{Numerical Implementation} \label{numerical}
	In Subsection \ref{newton}, we prove Theorem \ref{2.1.9.1}. The second subsection discusses results of numerical experiments that justify the theoretical estimates. 
	\subsection{Convergence of Newton's method}\label{newton}
	Now we establish that Newton iterates in \eqref{eq:3.1.1} converges quadratically to the discrete solution. The proof follows by modification of the approach used in \cite{Gaurang}. While the linearized operator in \cite{Gaurang} has a system of bilinear and trilinear forms; in this case we have system of bilinear and quadrilinear forms that leads to modification in the bounds. Moreover, the choice of the radius of the ball $\rho_1,$ in which the initial guess $\Psi_{\rm dG}^0$ needs to be chosen so as the Newton's method converges depends on the non-linearity and needs to be chosen carefully. The effect of the parameter $'\epsilon'$ is also considered in this proof.   		
	\begin{proof}[\textbf{Proof of Theorem \ref{2.1.9.1}}]
		Following the proof of 
		Lemma \ref{2.4.12}, for a sufficiently small choice of the discretization parameter $h=O(\epsilon^2 ),$
		 there exists a sufficiently small constant $\delta > 0$ independent of $h$ such that for each $Z_{\dg} \in\V_h$ that satisfies $\vertiii{Z_{\dg} -\textrm{I}_{\dg}\Psi}_\dg \leq \delta$,  the bilinear form 
		\begin{align}\label{2.1.9.2}
	\dual{{\rm DN}_h(Z_{\dg}) \Theta_{\dg},\Phi_{\dg}}=A_\dg(\Theta_{\dg},\Phi_{\dg})+3B_\dg(Z_{\dg},Z_{\dg},\Theta_{\dg},\Phi_{\dg} )+ C_\dg(\Theta_{\dg},\Phi_{\dg})
		\end{align}
		satisfies discrete inf-sup condition given by
		\begin{align}\label{6.2}
		0< \frac{\beta_0}{2} \leq \inf_{\substack{\Theta_{\rm dG} \in \V_h \\  \vertiii{\Theta_{\rm dG}}_{\rm dG}=1}} \sup_{\substack{\Phi_{\rm dG} \in \V_h \\  \vertiii{\Phi_{\rm dG}}_{\rm dG}=1}}\dual{{\rm DN}_h(Z_{\dg}) \Theta_{\dg},\Phi_{\dg}}.
		\end{align}
		 For a sufficiently small choice of the discretization parameter $h=O(\epsilon^{2+\alpha} )$ with $\alpha>0$, \eqref{5.9} leads to  
		\begin{align}\label{2.1.9.3}
		\vertiii{\Psi_{\dg}-\textrm{I}_{\dg}\Psi}_\dg\leq 2C_3h \leq \frac{\delta}{2}.
		\end{align}  
		Assume that the initial guess $\Psi_{\dg }^0$ satisfies $\vertiii{\Psi_{\dg}^0-\Psi_{\dg}}_\dg \leq \rho_1  \leq \frac{\delta}{2}.$
		A use of this and \eqref{2.1.9.3} along with a triangle inequality leads to
		\begin{align}\label{2.1.9.4}
		\vertiii{\Psi_{\dg}^0-\textrm{I}_{\dg}\Psi}_\dg \leq \vertiii{\Psi_{\dg}^0-\Psi_{\dg}}_\dg+ \vertiii{\Psi_{\dg}-\textrm{I}_{\dg}\Psi}_\dg  \leq \delta.
		\end{align}
        Therefore, $Z_\dg= \Psi_{\dg}^0$ in \eqref{2.1.9.2} leads to the discrete inf-sup condition for $\dual{DN_h(\Psi_{\dg}^{0})\cdot,\cdot }$, and this implies that there exists a unique $\Psi_{\dg }^1 \in \V_h$ ( which is the first Newton iterate in \eqref{eq:3.1.1} ) satisfying the well-posed system 
        \begin{align}\label{2.1.9.5}
        \dual{{\rm DN}_h(\Psi_{\dg}^{0}) \Psi_{\dg}^1,\Phi_{\dg} }
        =2B_{\rm dG}(\Psi_{\dg}^{0},\Psi_{\dg }^{0},\Psi_{\dg }^{0},\Phi_{\dg})+L_{\rm dG}(\Phi_{\rm dG})  \text{ for all } \Phi_{\dg} \in \V_h .
        \end{align} 
        The discrete inf-sup condition \eqref{6.2} implies the existence of a $\Phi_{\dg} \in \V_h$ with $\vertiiidg{\Phi_{\dg}}=1$ such that 
        \begin{align}\label{6.5}
        \frac{\beta_0}{2} \vertiii{\Psi_{\dg}^1-\Psi_{\dg}}_\dg \leq \dual{{\rm DN}_h(\Psi_{\dg}^{0}) (\Psi_{\dg}^1-\Psi_{\dg}),\Phi_{\dg} }.
        \end{align} 
        A use of \eqref{2.1.9.2}, \eqref{2.1.9.5} and \eqref{2.3.13} in \eqref{6.5} yields 
        \begin{align}\label{6.6}
        \frac{\beta_0}{2} \vertiii{\Psi_{\dg}^1-\Psi_{\dg}}_\dg &\leq 2B_{\dg}(\Psi_{\dg}^{0},\Psi_{\dg}^{0},\Psi_{\dg}^{0},\Phi_{\dg})+L_{\rm dG}(\Phi_{\rm dG}) -A_{\dg}(\Psi_{\dg},\Phi_{\dg})
        \notag\\ &\quad -3B_{\dg}(\Psi_{\dg}^{0},\Psi_{\dg}^{0},\Psi_{\dg},\Phi_{\dg})-C_{\dg}(\Psi_{\dg},\Phi_{\dg}) \notag			
        \notag\\ &=2B_{\dg}(\Psi_{\dg}^{0},\Psi_{\dg}^{0},\Psi_{\dg}^{0},\Phi_{\dg})-3B_{\dg}(\Psi_{\dg}^{0},\Psi_{\dg}^{0},\Psi_{\dg},\Phi_{\dg})\notag\\&\quad+B_{\dg}(\Psi_{\dg},\Psi_{\dg},\Psi_{\dg},\Phi_{\dg}).
        \end{align}
        The right hand side of \eqref{6.6} is estimated analogous to $T_4$ in Theorem \ref{thm2.5.1}. This followed by a use of the triangle inequality and \eqref{2.1.9.3} leads to
        \begin{align*}
        \vertiii{\Psi_{\dg}^1-\Psi_{\dg}}_\dg
        &\lesssim \epsilon^{-2}  \vertiii{\Psi_{\dg}^{0}-\Psi_{\dg}}_\dg^2(\vertiii{\Psi_{\dg}^{0}-\Psi_{\dg}}_\dg+ \vertiii{\Psi_{\dg}^{0}}_\dg)\\& \lesssim \epsilon^{-2}  \vertiii{\Psi_{\dg}^{0}-\Psi_{\dg}}_\dg^2(2\vertiii{\Psi_{\dg}^{0}-\Psi_{\dg}}_\dg+ \vertiii{\Psi_{\dg}-\textrm{I}_{\dg}\Psi}_\dg+\vertiii{\textrm{I}_{\dg}\Psi}_\dg)\notag\\&\lesssim \epsilon^{-2}  \vertiii{\Psi_{\dg}^{0}-\Psi_{\dg}}_\dg^2(\rho_1+h +1)\leq C_5\epsilon^{-2}  \vertiii{\Psi_{\dg}^{0}-\Psi_{\dg}}_\dg^2(\rho_1 +1), 
        \end{align*}
        where the constant $C_5$ depends on $\vertiii{\Psi}_2,$ $\beta_0$, $C_3,$ $C_P$ and $C_S.$ A choice of $\rho_1 < \min(\frac{\delta}{2},\frac{-1+\sqrt{1+2\epsilon^2C_5^{-1}}}{2}) $ yields
        $
        \vertiii{\Psi_{\dg}^1-\Psi_{\dg}}_{\dg}
        \leq
        \frac{1}{2}\vertiii{\Psi_{\dg}^{0}-\Psi_{\dg}}_{\dg}
        \leq  \frac{\rho_1}{2}$. A use of this estimate, \eqref{2.1.9.3} and triangle inequality yields  $
        \vertiii{\Psi_{\dg}^1-\textrm{I}_{\dg}\Psi}_{\dg}
        \leq
        \vertiii{\Psi_{\dg}^1-\Psi_{\dg}}_{\dg}+ \vertiii{\Psi_{\dg}-\textrm{I}_{\dg}\Psi}_{\dg}
        \leq \delta .$
         Also, $
        \vertiii{\Psi_{\dg}^1-\Psi_{\dg}}_{\dg}
        \leq C_6\vertiii{\Psi_{\dg}^{0}-\Psi_{\dg}}_{\dg}^2 
        ,$ where $C_6$ depends on $C_5$, $\rho_1$ and $\epsilon$.         
        \medskip
        
		\noindent Therefore, $\vertiii{\Psi_{\dg}^{j}-\textrm{I}_{\dg}\Psi}_\dg \leq \delta$ and $\vertiii{\Psi_{\dg}^j-\Psi_{\dg}}_{\dg}\leq\frac{\rho_1}{2^j}$  satisfies for $j=1.$  Assume that this holds for some $j\in \mathbb{N}.$ Then, $Z_\dg= \Psi_{\dg}^{j}$ in \eqref{2.1.9.2} leads to the discrete inf-sup condition for $\dual{{\rm DN}_h(\Psi_{\dg}^{j})\cdot,\cdot }$ which implies that there exists a unique $\Psi_{\dg }^{j+1} \in \V_h$ satisfying the well-posed system 
		\begin{align*}
		\dual{{\rm DN}_h(\Psi_{\dg}^{j}) \Psi_{\dg}^{j+1},\Phi_{\dg} }
		=2B_{\rm dG}(\Psi_{\dg}^{j},\Psi_{\dg }^{j},\Psi_{\dg }^{j},\Phi_{\dg})+L_{\rm dG}(\Phi_{\rm dG})  \text{ for all } \Phi_{\dg} \in \V_h .
		\end{align*}
Then, a use of \eqref{6.2} and  following the proof for $j=1$, we obtain
\begin{align*}
		\vertiii{\Psi_{\dg}^{j+1}-\Psi_{\dg}}_\dg
		\leq C_5\epsilon^{-2}  \vertiii{\Psi_{\dg}^{j}-\Psi_{\dg}}_\dg^2(\rho_1 +1). 
		\end{align*}
 		Since $\rho_1 < \min(\frac{\delta}{2},\frac{-1+\sqrt{1+2\epsilon^2C_5^{-1}}}{2}) $,
		$
		\vertiii{\Psi_{\dg}^{j+1}-\Psi_{\dg}}_{\dg}
		\leq
		\frac{1}{2}\vertiii{\Psi_{\dg}^{j}-\Psi_{\dg}}_{\dg}
		\leq  \frac{\rho_1}{2^{j+1}} $ and $
		\vertiii{\Psi_{\dg}^{j+1}-\textrm{I}_{\dg}\Psi}_{\dg}
		\leq\delta $  with a quadratic convergence rate given by
		$
		\vertiii{\Psi_{\dg}^{j+1}-\Psi_{\dg}}_{\dg}
		\leq C_6\vertiii{\Psi_{\dg}^{j}-\Psi_{\dg}}_{\dg}^2 
		.$  
		This completes the proof using mathematical induction.
	\end{proof} 
\subsection{Numerical experiments}\label{numerical experiment}
In this subsection, the computational error and convergence rate of discrete solutions for dGFEM are illustrated for some benchmark problems. We study the convergence of discrete solutions, for various values of $\epsilon$. These numerical experiments have been implemented using FEniCS \cite{fenics} library and the results support the theoretical findings (for the details of implementation procedure see \ref{algorithm}). Let $e_i$ and $h_i$ be the error and the mesh parameter at the $i$-th level, respectively. The $i$-th level experimental order of convergence is defined by $\displaystyle \alpha_i:=log(e_n/e_{i})/log(h_n/h_{i})$ for $ i=1, \ldots , n-1$ and $n $ corresponds to the final iteration considered in numerical experiments.
\begin{example}\label{Polynomial_example}
	For the problem \eqref{2.3.1.2}, set $\Omega=(0,1)\times (0,1)$ and the parameter value $\epsilon=0.2$. Compute the right hand side  for the manufactured exact solution  $u=x(1-x)y(1-y)$ and $v=x(1-x)y(1-y).$ 
\end{example}
\noindent We discretise the domain $\Omega=(0,1)\times (0,1)$ into triangles and in the uniform refinement process, each triangle $T$ is divided into four similar triangles. 
Tables \ref{P1_polynomial}, \ref{P2_polynomial} and \ref{P3_polynomial} present the numerical errors and orders of convergence in energy and $\mathbf{L}^2$ norms computed using piecewise polynomials of degree $1,$ $2$ and $3$, respectively.
\begin{table}[H]
	\centering
	\begin{tabular}{||c|c |c| c| c|c||} 
		\hline
		$h$    &      $\vertiiidg{\Psi-\Psi_{\dg}^{i}}$ &  Order &  $\vertiii{\Psi-\Psi_{\dg}^{i}}_{\mathbf{L}^2}$ &  Order  \\ [0.5ex] 
		\hline\hline
		0.3535      &  0.69481292E-1   &1.10439646  & 0.37102062E-2   & 2.13551126 \\	
		\hline
		0.1767     &  0.29094563E-1   & 1.06282552  & 0.74601448E-3 & 2.08646636  \\		
		\hline
		0.0883      &  0.13383488E-1  & 1.04108837  & 0.16427963E-3  & 2.04993501 \\
		\hline
		0.0441     &  0.64047595E-2 &  1.02761427 & 0.38801373E-4 &  2.03044486 \\
		\hline
		0.0220    &   0.31296010E-2 & 1.01899219 &  0.94539471E-5  & 2.02007895 \\
		\hline
		0.0141     &  0.19860395E-2 & - & 0.38377918E-5 &  -  \\
		\hline
	\end{tabular}
	\caption{ Numerical errors and orders of convergence in $\dg$ and $\mathbf{L}^2$ norms using piecewise linear polynomials.  
	}
\label{P1_polynomial}		
\end{table}	
\begin{table}[H]
	\centering
	\begin{tabular}{||c|c |c| c| c|c||} 
		\hline
		$h$    &      $\vertiiidg{\Psi-\Psi_{\dg}^{i}}$ &  Order &  $\vertiii{\Psi-\Psi_{\dg}^{i}}_{\mathbf{L}^2}$ &  Order  \\ [0.5ex] 
		\hline\hline
		0.3535      &  0.20812609E-1   &2.25360832 & 0.59563518E-3   & 3.22432824  \\	
		\hline
		0.1767      &  0.35369466E-2   & 2.17037967  & 0.50612758E-4   & 3.13307115  \\		
		\hline
		0.0883      &  0.71656832E-3  & 2.12009381  & 0.53034861E-5   & 3.08714692  \\
		\hline
		0.0441     &  0.15828049E-3 &  2.08449035  & 0.60454668E-6 &  3.05924399  \\
		\hline
		0.0220     &   0.36909294E-4 & 2.05973780 & 0.71945319E-7  &  3.04117003  \\
		\hline
		0.0141     &  0.14720321E-4 & - & 0.18516670E-7  &  -   \\
		\hline
	\end{tabular}
	\caption{ Numerical errors and orders of convergence in $\dg$ and $\mathbf{L}^2$ norms using piecewise quadratic polynomials.  
	}
	\label{P2_polynomial}
\end{table}	
\begin{table}[H]
	\centering
	\begin{tabular}{||c|c |c| c| c|c||} 
		\hline
		$h$    &      $\vertiiidg{\Psi-\Psi_{\dg}^{i}}$ &  Order &  $\vertiii{\Psi-\Psi_{\dg}^{i}}_{\mathbf{L}^2}$ &  Order  \\ [0.5ex] 
		\hline\hline
	   0.3535   &  0.44002858E-2   &2.76328893  & 0.79005568E-4   & 3.7473338  \\	
		\hline
		0.1767    &  0.81936582E-3   &2.85612226  & 0.76031244E-5   & 3.8488866 \\		
		\hline
		0.0883      &  0.12772385E-3  & 2.92217824 & 0.59938677E-6   & 3.9184264  \\
		\hline
		0.0441    &  0.17487379E-4 & 2.95474574  & 0.41209120E-7 & 3.9524773 \\
		\hline
		0.0220    &   0.22702291E-5 & 2.96925442 & 0.26800851E-8  & 3.9677981  \\
		\hline
		0.0141     &  0.60334919E-6 & - & 0.45615227E-9  &  -   \\
		\hline
	\end{tabular}
	\caption{ Numerical errors and orders of convergence in $\dg$ and $\mathbf{L}^2$ norms using piecewise cubic polynomials.  
	} 
	\label{P3_polynomial}		
\end{table}	
\begin{rem}
This numerical example verifies that  the theoretical convergence rates obtained in energy norm (resp. $\mathbf{L}^2$ norm)  are $1$, $2$ and $3$ (resp. $2,$ $3$ and $4$) for piecewise $P_1$, $P_2$ and $P_3$ polynomial approximations. Improved rate of convergences suggest an improvement in the theoretical estimates if the exact solution is smooth. The convergence analysis for $h$-$p$ dGFEM taking into account the effect of $h$-$\epsilon$-$p$ dependency is a problem of future interest.
\end{rem}
\begin{example}\label{Luo_exampple}
	Consider the problem \eqref{2.3.1.2} in $\Omega=(0,1)\times (0,1).$ We  will approximate the system \eqref{2.3.1.2} using the Dirichlet boundary condition  \cite{MultistabilityApalachong} given by 
	\begin{equation*} 
	\mathbf{g}=
	\begin{cases} 
	(\textit{T}_{d}(x),0) & \text{on}\,\,\,\, y=0 \,\,\,\,\text{and} \,\,\,\, y=1, \\
	(- \textit{T}_d(y),0)  & \text{on}\,\,\,\, x=0 \,\,\,\,  \text{and} \,\,\,\,  x=1, 
	\end{cases}
	\end{equation*}  
	where the parameter $d=3 \epsilon$ and the trapezoidal shape function $\textit{T}_d:[0,1]\rightarrow {\mathbb{R}}$ is defined to be
	\begin{equation} \label{trapizoidal function}
	\textit{T}_d(t)=
	\begin{cases} 
	t/d, & 0 \leq t \leq d, \\
	1, &  d \leq  t \leq 1- d,  \\
	(1-t)/d, & 1- d \leq t \leq 1.
	\end{cases}
	\end{equation}
\end{example}
\noindent This example is motivated by the benchmark problem studied in \cite{MultistabilityApalachong}. There are two classes of stable experimentally observable configurations  [see Figure \eqref{fig diagonal:}]: the
diagonal states for which the director is aligned along the cell/square diagonals and the rotated states, for which the directors rotate by $\pi$-radians across the width of cell.	
\begin{figure}[H]
	\centering
	\subfloat[]{\includegraphics[scale = 0.35]{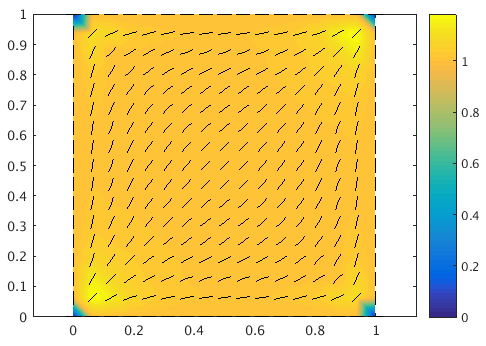}\label{D1}} 
	\hspace{0.4cm}
	\subfloat[]{\includegraphics[scale = 0.35]{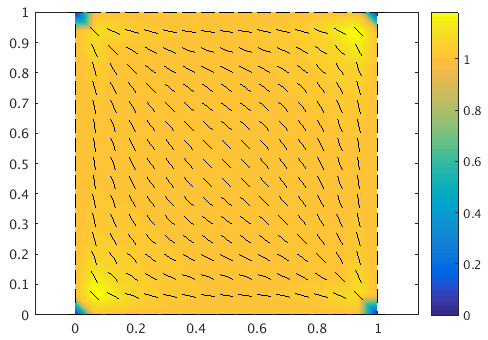}\label{D2}}
	\hspace{0.4cm}
	\subfloat[]{\includegraphics[scale = 0.35]{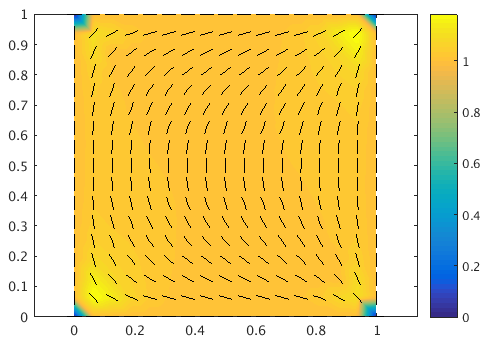}\label{R1}} 
	\hspace{0.4cm}
	\subfloat[]{\includegraphics[scale = 0.35]{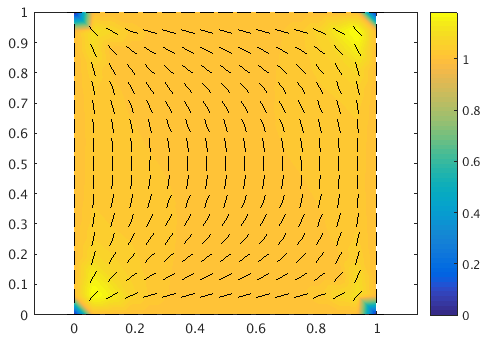}\label{R2}}
	\hspace{0.4cm}
	\subfloat[]{\includegraphics[scale = 0.35]{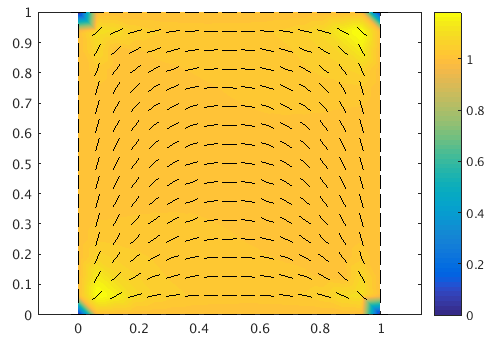}\label{R3}} 
	\hspace{0.4cm}
	\subfloat[]{\includegraphics[scale = 0.35]{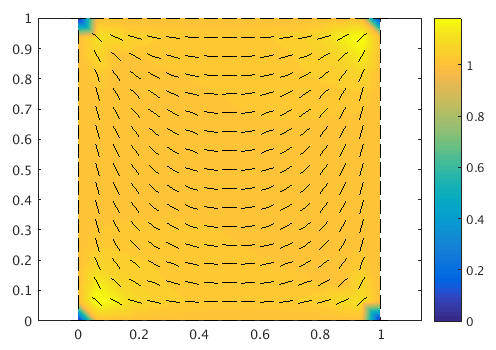}\label{R4}}
	\caption{ Diagonally stable molecular alignments: \protect \subref{D1} D1, \protect \subref{D2} D2 states and rotated stable molecular alignments: \protect \subref{R1} R1, \protect \subref{R2} R2, \protect \subref{R3} R3, \protect \subref{R4} R4 states.}
	\label{fig diagonal:}
\end{figure} 
\begin{tiny}
	\begin{table}[H]
		\centering
		\begin{tabular}{|c | c | c | c | c |}
			\hline
			Solution & $x=0$ & $x=1$ & $y=0$ & $y=1$\\ 
			\hline
			D1 & $\pi/2$ & $\pi/2$ & 0 & 0 \\
			D2 & $\pi/2$ & $\pi/2$ & $\pi$ & $\pi$ \\
			R1 & $\pi/2$ & $\pi/2$ & $\pi$ & 0 \\
			R2 & $\pi/2$ & $\pi/2$ & 0 & $\pi$ \\
			R3 & $3\pi/2$ & $\pi/2$ & $\pi$ & $\pi$ \\
			R4 & $\pi/2$ & $3\pi/2$ & $\pi$ & $\pi$ \\
			\hline
		\end{tabular}
		\caption{Initial conditions for Newton's method.}
		\label{table:initial}
	\end{table}
\end{tiny}
We use Newton's method to approximate the six different solutions $\Psi=(u,v)$ corresponding to the D1, D2, R1, R2, R3 and R4 states. We compute six different initial conditions by solving the Laplace equation \cite{MultistabilityApalachong} with Dirichlet  boundary conditions as specified in Table \ref{table:initial}. For example, in the D1 case, the system  $\displaystyle \Delta\theta=0 \text{ in } \Omega, \theta(x,0)=\theta(x,1)=0 \text{ and } \theta(0,y)=\theta(1,y)=\frac{\pi}{2}$ is solved using dGFEM. Then the corresponding initial condition for Newton's method (for the Landau-de Gennes mdodel) is defined to be
\begin{align}\label{initial_guess}
\Psi^0_{\rm dG}= s(\cos2\theta,\sin2\theta),
\end{align}
where $s=1$ at the interior nodes and $s=\abs{\mathbf{g}}$ at the boundary nodes.
We obtain the expected theoretical convergence rates using piecewise linear polynomial in dG and $\mathbf{L}^2$ norm as $O(h)$ and $O(h^2)$, respectively. 
 In Tables \ref{table:2} and \ref{table:3}, we record the computational errors and convergence rates of solutions for the diagonal state D1 and rotated state R1 respectively for $\epsilon= 0.02$. Similar errors and optimal convergence rates are obtained for D2, R2, R3 and R4 solutions. In Figure~\ref{fig diagonal:}, we plot the converged director plots and scalar order parameter for the six states, D1, D2, R1, R2, R3 and R4, respectively. In Figures~\ref{solution_D} - \ref{solution_R34}, we plot the level curves of the corresponding converged D1 and D2 diagonal, R1, R2, R3 and R4 rotated solutions.
\begin{table}[H]
	\centering
	\begin{tabular}{||c|c |c| c| c|c||} 
		\hline
		$h$    &     Energy    &  $\vertiiidg{\Psi_{\dg}^{n}-\Psi_{\dg}^{i}}$ &  Order  & $\vertiii{\Psi_{\dg}^{n}-\Psi_{\dg}^{i}}_{\mathbf{L}^2}$ &  Order  \\ [0.5ex] 
		\hline\hline
		0.0883      & 77.80650525 & 3.42185356  &0.90770236  & 0.23846079E-1   & 1.75592834 \\
		\hline
		0.0441     &  77.90383430 &  1.90966722 &0.94082516   & 0.78385873E-2 & 1.83134924   \\
		\hline
		0.0220     & 77.92229141   &1.00706128 & 0.95848055 & 0.24465831E-2  & 1.98287313   \\
		\hline
		0.0110     & 77.94112012  &  0.51823233 &- & 0.61895017E-3  & -  \\
		\hline
	\end{tabular}
	\caption{ Numerical energy, errors and convergence rates for  D1 solution in $\dg$ and $\mathbf{L}^2$ norm. 
	}
	\label{table:2}
\end{table}

\begin{table}[H]
	\centering
	\begin{tabular}{||c|c |c| c| c|c||} 
		\hline
		$h$    &     Energy    &  $\vertiiidg{\Psi_{\dg}^{n}-\Psi_{\dg}^{i}}$ &  Order &  $\vertiii{\Psi_{\dg}^{n}-\Psi_{\dg}^{i}}_{\mathbf{L}^2}$ &  Order  \\ [0.5ex] 
		\hline\hline
		0.0883     & 86.44084273 &3.46656699   &0.90994382   & 0.26417589E-1   & 1.79599363  \\
		\hline
		0.0441  &  86.53931303 &  1.92787441 &  0.94166745& 0.82372577E-2 &  1.85335905  \\
		\hline
	    0.0220   & 86.55785529  &1.01547693 & 0.95848131& 0.25178324E-2  &  1.99673627   \\
		\hline
		0.0110     & 86.57670525  &0.52256273 & -& 0.63088371E-3  &   -   \\
		\hline
	\end{tabular}
	\caption{ Numerical energy, errors and convergence rates for  R1 solution in $\dg$ and $\mathbf{L}^2$ norm. 
	}
	\label{table:3}		
\end{table}
\begin{figure}[H]
	\centering
	\subfloat[]{\includegraphics[scale = 0.20]{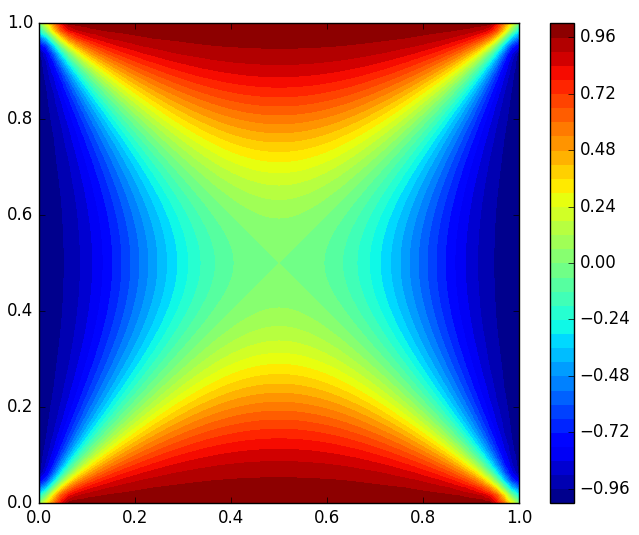}\label{uD1}} 
	\hspace{0.2cm}
	\subfloat[]{\includegraphics[scale = 0.20]{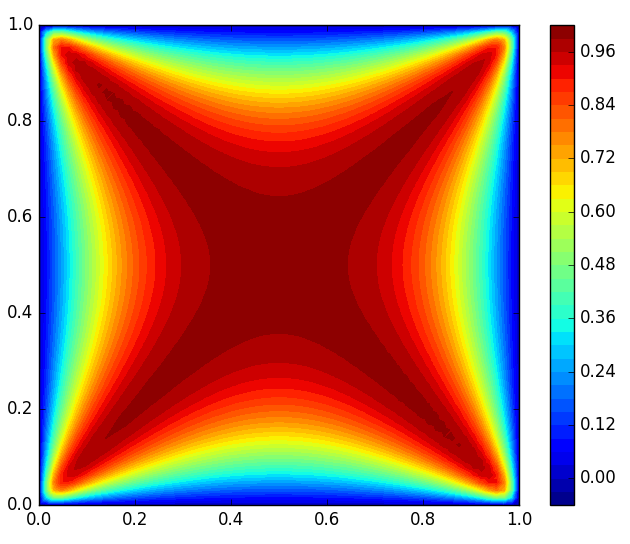}\label{vD1}}
	\hspace{0.2cm}
	\subfloat[]{\includegraphics[scale = 0.20]{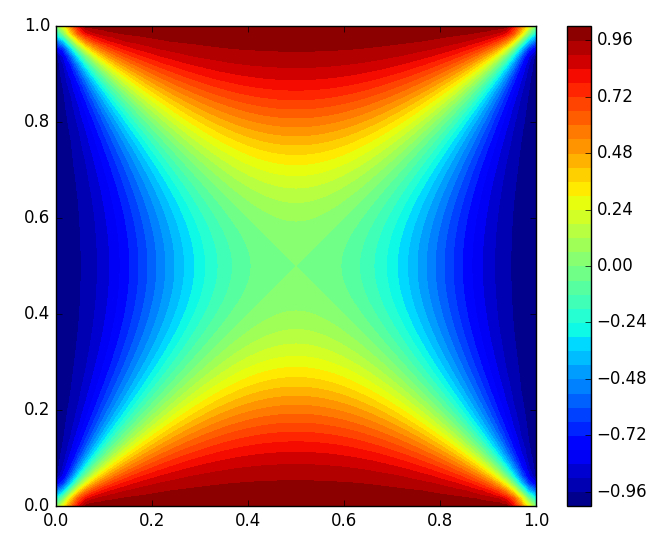}\label{uD2}} 
	\hspace{0.2cm}
	\subfloat[]{\includegraphics[scale = 0.20]{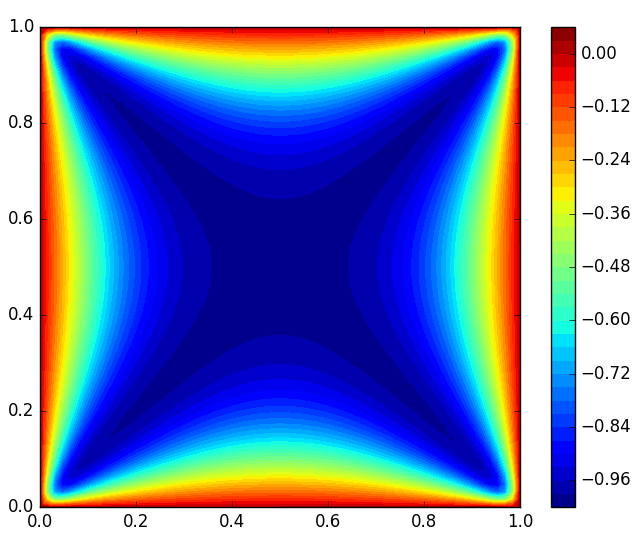}\label{vD2}}
	\caption{ D1 solution $\Psi_{\rm dG}$ components: \protect \subref{uD1} $u_{\rm dG}$, \protect\subref{vD1} $v_{\rm dG}$ and D2 solution $\Psi_{\rm dG}$ components: \protect \subref{uD2} $u_{\rm dG}$, \protect \subref{vD2} $v_{\rm dG}$. }
	\label{solution_D}
\end{figure} 
\begin{figure}[H]
	\centering
	\subfloat[]{\includegraphics[scale = 0.20]{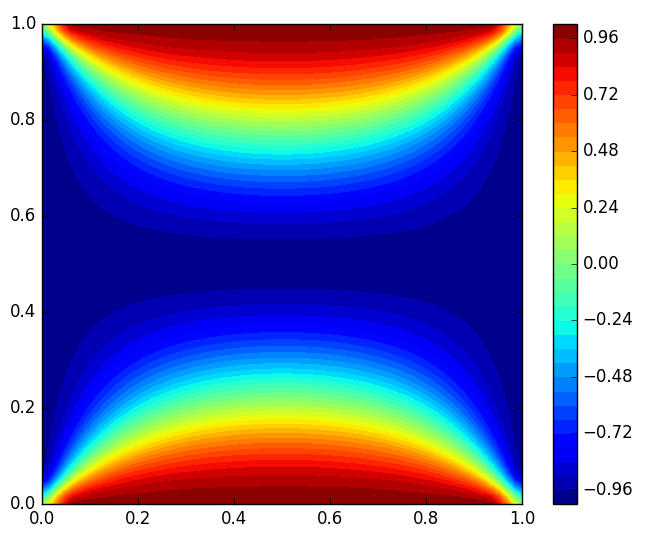}\label{uR1}} 
	\hspace{0.2cm}
	\subfloat[]{\includegraphics[scale = 0.20]{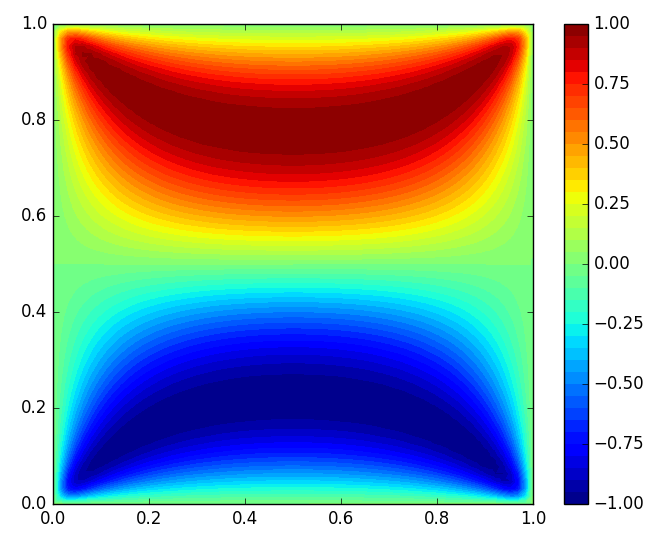}\label{vR1}}
	\hspace{0.2cm}
	\subfloat[]{\includegraphics[scale = 0.20]{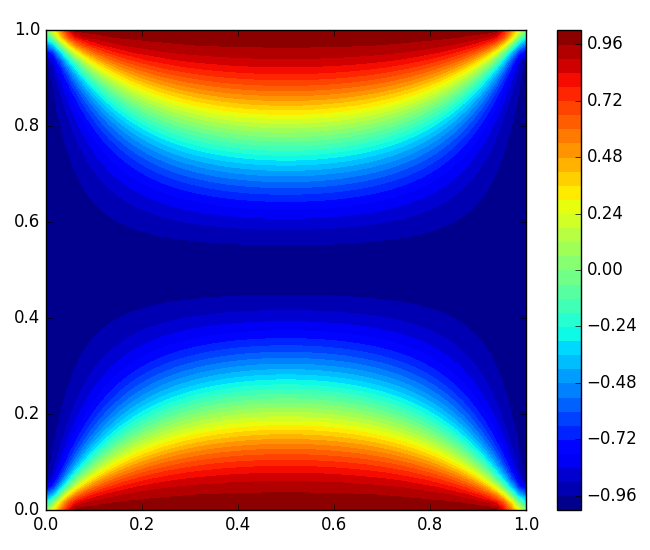}\label{uR2}} 
	\hspace{0.2cm}
	\subfloat[]{\includegraphics[scale = 0.20]{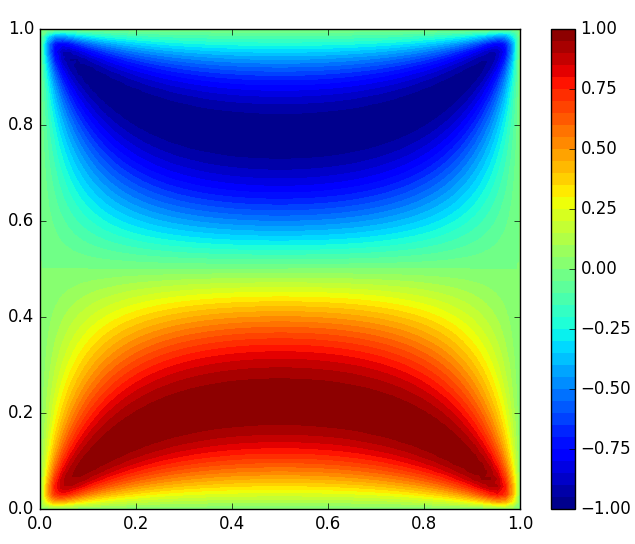}\label{vR2}}
	\caption{ R1 solution $\Psi_{\rm dG}$ components: \protect \subref{uR1} $u_{\rm dG}$, \protect\subref{vR1} $v_{\rm dG}$ and R2 solution $\Psi_{\rm dG}$ components: \protect \subref{uR2} $u_{\rm dG}$, \protect \subref{vR2} $v_{\rm dG}$. }
	\label{solution_R12}
\end{figure} 
\begin{figure}[H]
	\centering
	\subfloat[]{\includegraphics[scale = 0.20]{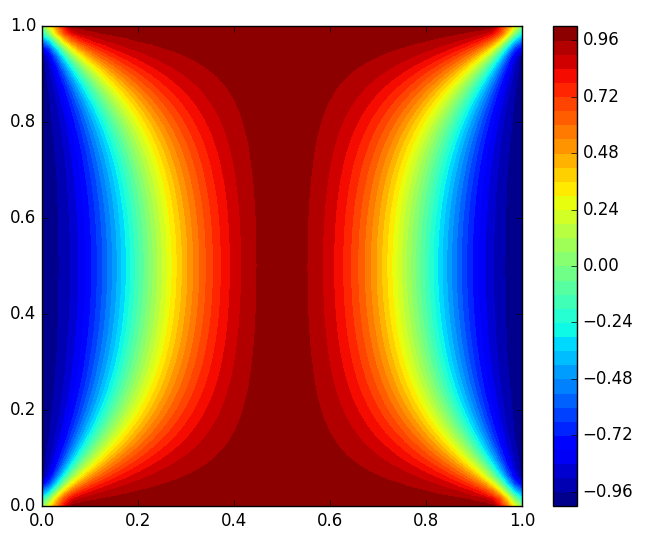}\label{uR3}} 
	\hspace{0.2cm}
	\subfloat[]{\includegraphics[scale = 0.20]{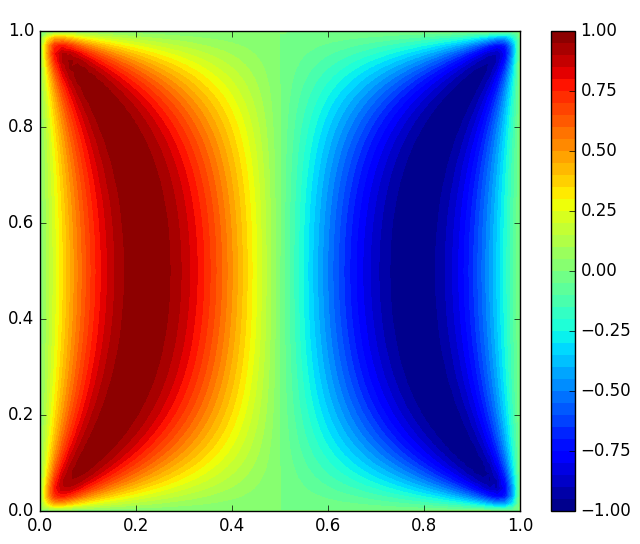}\label{vR3}}
	\hspace{0.2cm}
	\subfloat[]{\includegraphics[scale = 0.20]{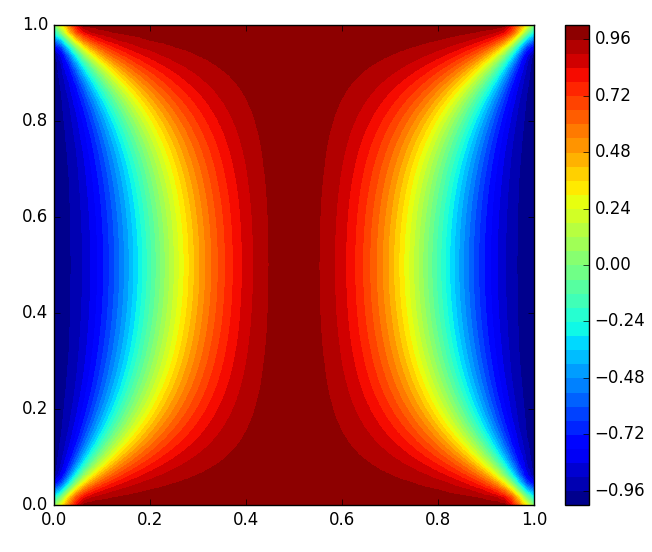}\label{uR4}} 
	\hspace{0.2cm}
	\subfloat[]{\includegraphics[scale = 0.20]{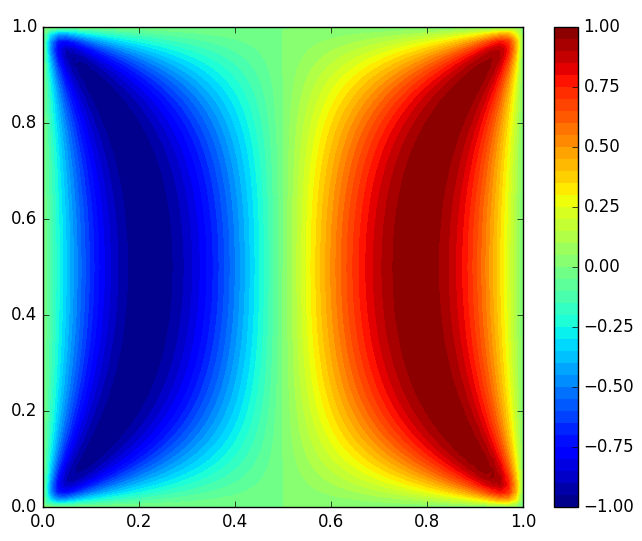}\label{vR4}}
	\caption{ R3 solution $\Psi_{\rm dG}$ components: \protect \subref{uR3} $u_{\rm dG}$, \protect\subref{vR3} $v_{\rm dG}$ and R4 solution $\Psi_{\rm dG}$ components: \protect \subref{uR4} $u_{\rm dG}$, \protect \subref{vR4} $v_{\rm dG}$. }
	\label{solution_R34}
\end{figure} 	
\begin{figure}[h!]
	\centering
	\includegraphics[height=7cm, width=8cm]{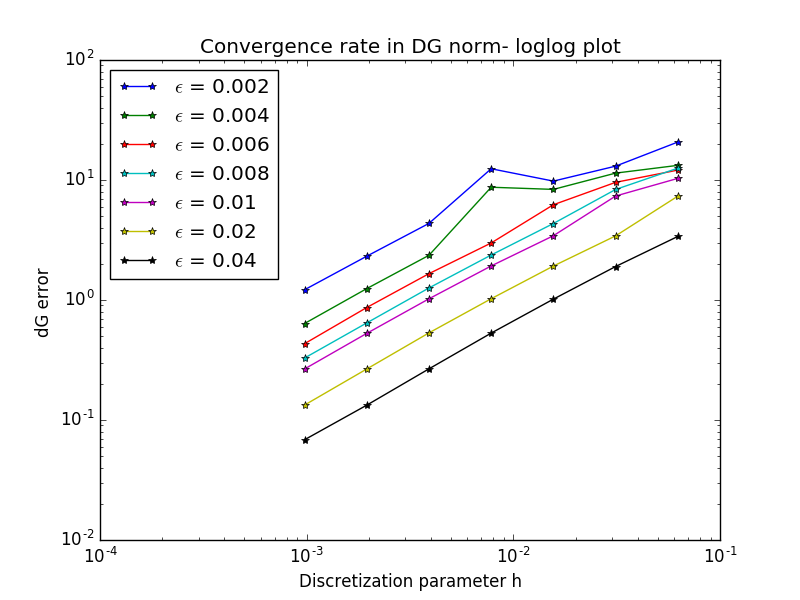}
	\caption{Behavior of errors for different values of discretization parameter $h$ for various values of $\epsilon$.  }
	\label{epsilon_dependency}
\end{figure}
Figure \ref{epsilon_dependency} plots the dG error vs the discretization parameter for various values of $\epsilon$. Note that errors are sensitive to the choice of discretization parameter as $\epsilon$ decreases. 
It is difficult to verify the exact dependence of $h$ and $\epsilon$ from these numerical results, except that as $\epsilon \rightarrow 0$, smaller mesh-sizes are needed for convergence.
\begin{rem}
	 In \cite{MultistabilityApalachong}, the authors study the model problem of nematic-filled square wells in the conforming FEM set up and present numerical errors and convergence rates for conforming FEM, in $\mathbf{H}^1$ and $\mathbf{L}^2$ norm to be of $O(h)$ and $O(h^2),$ respectively, for the six different stable states.  However, they do not study the convergence trends as a function of $\epsilon$ and do not establish the theoretical order of convergence.
\end{rem}
\begin{example}\label{Annular_example}
	Consider an annular domain filled with  the compound $5CB$, a standard liquid crystal material \cite{Dassbach_thesis}.
	We consider the weak formulation of the Euler-Lagrange equations for a reduced two-dimensional Landau-de Gennes energy as in  \cite{Dassbach_thesis}
	\begin{align}\label{annular_equation}
	-\Delta \Psi + (-1+2\mathcal{C} (\abs{\Psi}^2)\Psi = \mathbf{f} \text{ in }\Omega  \,\, \text{ and }\Psi= \mathbf{g} \text{ on } \partial \Omega,
	\end{align}
	where $\mathcal{C} = \frac{\tilde{C}}{\abs{A}}$ is a constant that depends on the material parameters $A$ and $\tilde{C}$ of the liquid crystal. The parameter values are $A= 0.172 \times 10^6 N/m^2$ and $\tilde{C}= 1.73 \times 10^6 N/m^2$ for the compound $5CB$.
	 The domain has an outer radius $L_0$ and the inner radius $0.5 L_0$, where $L_0$ is the characteristic length scale. The ratio ($\tilde{\rho}$) of the inner to the outer radius has been chosen to be $0.5$ to capture the radial solution analyzed by Bethuel \textit{et al.} \cite{Bethual_Colemn}. In \cite{Bethual_Colemn}, the authors study the Oseen-Frank radial solution in this annulus within the one-constant framework with energy given by $\mathcal{E}_{OF}(\mathbf{n})=K\int_{\Omega} \abs{\nabla \mathbf{n}}^2 \dx$, where $K>0$ is an elastic constant. The radial solution is simply given by $\mathbf{n} = \left( \frac{x}{\sqrt{x^2 + y^2}}, \frac{y}{\sqrt{x^2 + y^2}}\right)$.
	In the reduced Landau-de Gennes framework, the corresponding manufactured solution, $\Psi=(u,v)$, is defined by  
	$
	u=  \frac{2x^2}{(x^2 +y^2)} -1 \text{ and } v= \frac{2x y}{(x^2 +y^2)} 
	$, or equivalently $\Psi = \left(2\mathbf{n} \otimes \mathbf{n} - I \right)$, $\mathbf{n}$ is the Oseen-Frank radial solution studied in \cite{Dassbach_thesis} and $I$ is the $2\times 2$ identity matrix. 
	The data $\mathbf{f}$ and $ \mathbf{g}$ are calculated by substituting this manufactured solution into the left-hand side of \eqref{annular_equation}.
	\begin{figure}[H]
		\centering
		\subfloat[]{\includegraphics[scale = 0.21]{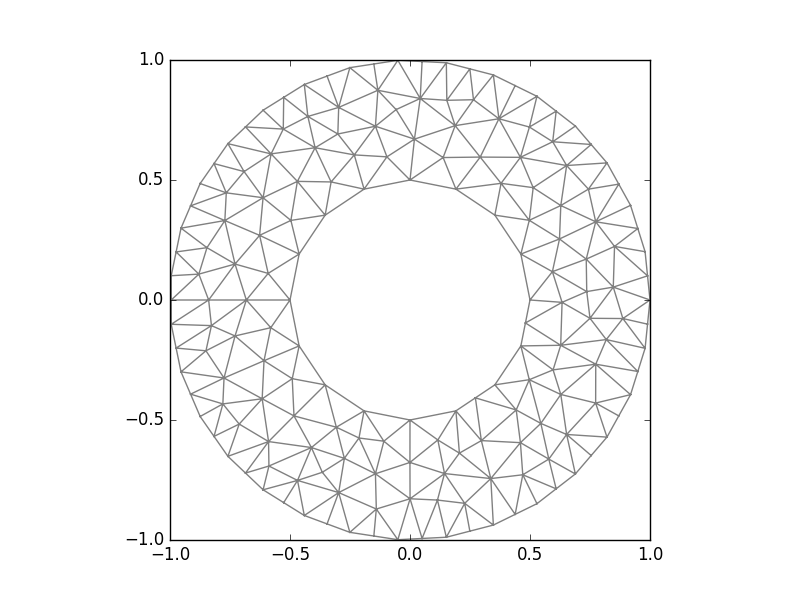} \label{Annularmesh}}
		\hspace{0.1cm}
		\subfloat[]{\includegraphics[scale = 0.15]{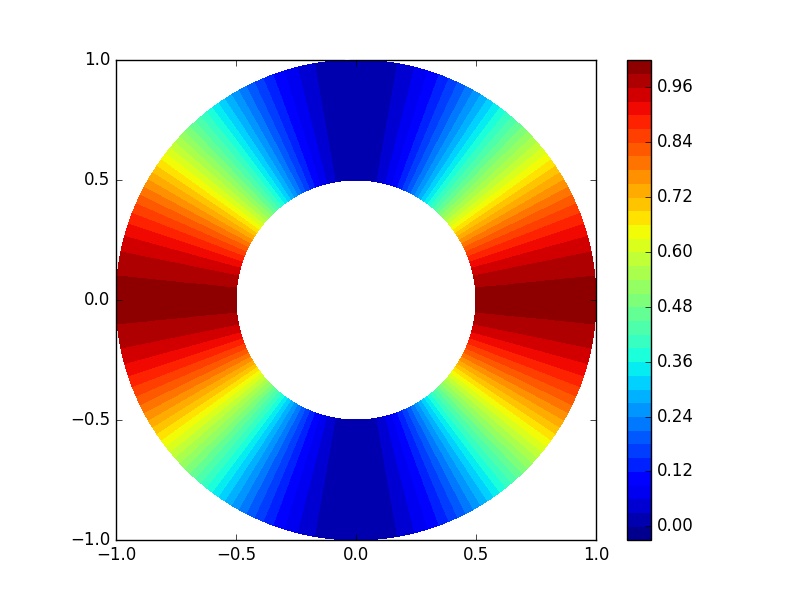}\label{u}} 
		\hspace{0.1cm}
		\subfloat[]{\includegraphics[scale = 0.15]{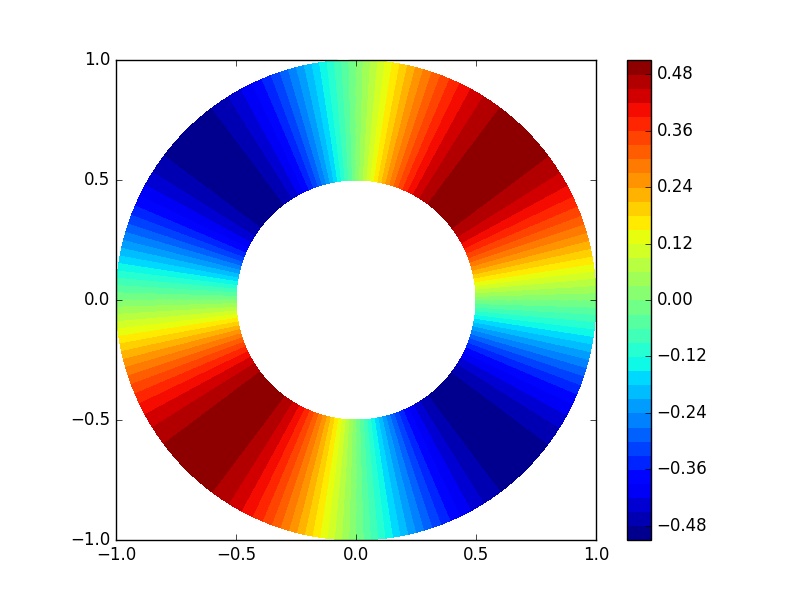}\label{v}}
		\caption{\protect \subref{Annularmesh}: Initial mesh for Annular domain with $\tilde{\rho}=0.5$,  solution components of $\Psi_{\rm dG}$ : \protect \subref{u} $u_{\rm dG}$, \protect\subref{v} $v_{\rm dG}$.
		}
		\label{solution_Annular}
	\end{figure} 	
\noindent We use dGFEM formulation based on piecewise linear polynomials over triangles, to study the convergence of the discrete solution to this manufactured solution. The mesh is generated using \textit{mshr}, the mesh generation component in FEniCS. The curved boundary is approximated using polygons. The errors and order of convergences in both energy and $\mathbf{L}^2$ norms  are tabulated in Table \ref{Annular_domain} and the results are consistent with our theoretical estimates. In Figure \ref{Annularmesh}, we plot the initial triangulation for the annular domain and in Figures~\ref{u} - \ref{v}, we plot the level curves of the corresponding converged radial solution.
\end{example}
\begin{table}[H]
	\centering
	\begin{tabular}{||c|c| c| c|c||} 
		\hline
		Max $h$   &      $\vertiiidg{\Psi-\Psi_{\dg}^{i}}$ &  Order &  $\vertiii{\Psi-\Psi_{\dg}^{i}}_{\mathbf{L}^2}$ &  Order  \\ [0.5ex] 
		\hline\hline
		0.195    &  0.19019778 & 1.02678772 & 0.70751443E-2  & 2.07685808    \\
		\hline
		0.999 E-1    &  0.89721415 E-1 & 1.01083542 & 0.11630918E-2  & 1.92959391    \\
		\hline
		0.333 E-1      &  0.29026972 E-1 &   0.99787440 & 0.13457407E-3  & 1.90301464 \\
		\hline
		0.199 E-1      &  0.17390702 E-1  & 0.99496176 & 0.49154393E-4  &  1.86299044  \\
		\hline
		0.142 E-1     &  0.12534273 E-1 & 1.00852040  & 0.26717823E-4 & 1.89491873  \\
		\hline                                                                 
		0.111 E-1    &  0.97316088 E-2  & - & 0.14469534E-4  & - \\
		\hline
	\end{tabular}
	\caption{ Numerical errors and convergence rates for solution of \eqref{annular_equation} in $\dg$ and $\mathbf{L}^2$ norm.
	}
	\label{Annular_domain}		
\end{table}		
\begin{rem}
	In \cite{Dassbach_thesis}, \eqref{annular_equation} is implemented using conforming FEM
 and the numerical orders of convergences in $\mathbf{L}^{\infty},$ $\mathbf{L}^2$ and $\mathbf{H}^1$ norms are observed to be of orders of $O(h^2),$ $ O(h^2) $ and $O(h)$, respectively. However, the theoretical error estimates are not derived in \cite{Dassbach_thesis}.  
\end{rem}
\section{Conclusion}
The paper focuses on a rigorous dG formulation of the model problem studied in \cite{MultistabilityApalachong, Tsakonas}. It is not clear if the dG formulation gives us any additional physical insight for this model problem, at least for micron-sized wells. However, dG methods combined with a posteriori error analysis generally work well for problems with less regularity and with non-homogeneous boundary conditions or for singular liquid crystal potentials as proposed by Katriel et.al \cite{Katriel_et_al} and studied in \cite{Ball}, or when $\epsilon \rightarrow 0$ (sharp interface limit). Hence, we expect our work to be foundational for subsequent cutting-edge numerical studies of interfacial phenomena, higher-dimensional defects, quenching phenomena for LCs which are necessarily characterized by singular behavior.  A particularly interesting application would be the  dG formulation of the continuum theory of ferronematics \cite{Ferromagnetic}. Ferronematics in confinement naturally exhibit domain walls and interior point defects and it would be interesting to see how dG results compare to numerical results from confoming FEM, particularly near structural defects e.g. nematic vortices and boundary layers.

	\bibliographystyle{amsplain}
	
	\section*{Acknowledgements}
	R.M. gratefully acknowledges support from institute Ph.D. fellowship and the Visiting Postgraduate Scholar scheme run by the University of Bath. A.M. gratefully acknowledges hospitality from IIT Bombay.
	A.M. is supported by an EPSRC Career Acceleration Fellowship EP/J001686/1 and EP/J001686/2, University of Bath Internationalisation schemes and an OCIAM Visiting Fellowship, the Keble Advanced Studies Centre. N.N. gratefully acknowledges support from the Faculty of Science at the University of Bath and the support by  DST SERB MATRICS grant MTR/2017/000 199. The authors gratefully
	acknowledge Prof. M. Vanninathan for the insightful discussions and also the HPC facility of IIT Bombay for the computational resources. The authors thank the referees for the constructive comments that has improved the quality of the paper significantly.

	\bibliography{ReferencesLDG}
\begin{appendices}	
	\section{Appendix}
	\subsection{Derivation of weak formulation}\label{Weak_formulation}
	The Landau-de Gennes energy functional considered in this article is given by 
	\begin{align*} 
	\mathcal{E}(\Psi) =\int_\Omega (\abs{\nabla \Psi}^2 +
	\epsilon^{-2}(\abs{\Psi}^2- 1)^2) \dx 
	\end{align*}
 with Dirichlet boundary condition $\Psi = \mathbf{g} $ on $\partial\Omega.$ Consider the real valued function $i[\tau]:=\mathcal{E}(\Psi + \tau \Phi) \, \, (\tau \in \mathbb{R})$
	with $\Phi\in \V.$ Since $\Psi$ is a minimizer of $\mathcal{E}(\cdot)$ and $\Psi + \tau \Phi =\Psi = \mathbf{g}$ on $\partial \Omega,$ $\mathcal{E}(\cdot)$ has a minimum at $\tau = 0.$ Therefore,  
		 $i'(0)=\lim_{\tau \rightarrow 0} \dfrac{\mathcal{E}(\Psi + \tau \Phi)-\mathcal{E}(\Psi )}{\tau} =0.$
	A manipulation using the definition of $\mathcal{E}(\cdot) $ leads to
	\begin{align*}
	\int_{\Omega} \nabla \Psi \cdot \nabla \Phi \dx + 2\epsilon^{-2} \int_{\Omega} (\abs{\Psi}^2-1)(\Psi \cdot \Phi)\dx=0 \quad \text{ for all } \Phi \in \V, 
	\end{align*}
	 which yields the weak formulation \eqref{eq:2.1.5} in operator form. The operator form in \eqref{eq:2.1.5} has a representation of the nonlinear part  in terms of $B(\cdot, \cdot, 
	 \cdot, \cdot)$ and this is crucial for an elegant analysis. 
	\subsection{Regularity result} \label{Regularity result}	
\begin{thm}\label{2.3.3.1} \cite{grisvard}  \cite{KesavaTopicsFunctinal}
	Let $\Omega$ be a convex, bounded and open subset of $\mathbb{R}^2$ with smooth boundary $\partial \Omega.$ 
	Then, for $F \in \mathbf{L}^2(\Omega)$ and $\mathbf{g} \in \mathbf{H}^{\frac{3}{2}}(\partial\Omega)$, there exists a $\Psi \in \mathbf{H}^2(\Omega)$ such that $-\Delta \Psi =F \text{ in } \Omega \text{ and }\Psi = \mathbf{g} \text{ on } \partial \Omega. $
\end{thm}
\begin{proof}[\textbf{Proof of Lemma \ref{regularity}}]
	Rewrite the system \eqref{2.3.1.2}  as: $	-\Delta \Psi =F_1(\Psi) \text{ in } \Omega  \text{ and } \Psi = \mathbf{g} \text{ on } \partial \Omega, $
	where $F_1(\Psi)= -\dfrac{2}{\epsilon^2}(\abs{\Psi}^2-1)\Psi$. 
	Expand the expression for $F_1(\Psi) $, where $\Psi=(u,v) \in \mathbf{H}^1(\Omega)$,  use the Sobolev embedding result $H^1(\Omega) \hookrightarrow L^p(\Omega) $ for all $p \geq 1$, and the 
	 H\"older's inequality to prove that $F_1(\Psi) \in \mathbf{L}^2(\Omega).$	 
	Now a use of Theorem \ref{2.3.3.1} and a bootstrapping argument \cite{Evance19} imply that the solution $\Psi $ of \eqref{2.3.1.2} belongs to $\mathbf{H}^{2}(\Omega)$. 
\end{proof}
\subsection{Details of proof of Theorem \ref{2.5.2}}\label{Brouwer}
\begin{thm}[Brouwer fixed point theorem]\label{Schauder}\cite{KesavaTopicsFunctinal}
Let $Y$ be a finite dimensional Hilbert space  and   $f: K \rightarrow K$ be a continuous map from a non-empty, compact and convex subset $K$ of $Y$ which maps into $K.$ Then $f $ has a fixed point on $K$.	
\end{thm}
A use of Lemma \ref{4.3} yields that $\mu_{\dg}$ is Lipschitz continuous on the ball $\mathbb{B}_{R(h)}({\rm I}_{\rm dG}\Psi)$.
 Set $K= \mathbb{B}_{R(h)}({\rm I}_{\rm dG}\Psi)$ and $Y=\V_h$. Theorem \ref{thm2.5.1} yields $\mu_{\dg}$ maps $K $ to itself. Therefore, an application of Theorem \ref{Schauder} yields that $\mu_{\dg}$  has a fixed point, say $\Psi_{\rm dG}$. Now we prove any fixed point of $\mu_{\dg}$ is a solution of \eqref{2.3.13}.
 Since $\Psi_{\rm dG }$ is a fixed point of the map $\mu_{\dg} $, we have 
\begin{align*}
\dual{{\rm DN}_h({{\rm I}_{\rm dG}\Psi}) \mu_{\dg}(\Psi_{\dg}),\Phi_{\rm dG} } = \dual{{\rm DN}_h({{\rm I}_{\rm dG}\Psi}) \Psi_{\dg},\Phi_{\rm dG} }\,\, \text{ for all } \Phi_{\rm dG}  \in \V_{h}.
\end{align*}
A use of \eqref{2.5.1} on the left hand side and definition of $\dual{DN_h({{\rm I}_{\rm dG}\Psi}) \cdot,\cdot }$ on the right hand side leads to 
\begin{align*}
 3B_\dg(\textrm{I}_{\dg}\Psi, \textrm{I}_{\dg}\Psi,\Psi_{\dg},\Phi_{\dg})   - B_\dg(\Psi_{\dg},\Psi_{\dg}, \Psi_{\dg},\Phi_{\dg}) +L_\dg(\Phi_{\dg}) &=A_{\rm dG}(\Psi_{\dg},\Phi_{\rm dG})+C_{\rm dG}(\Psi_{\dg},\Phi_{\rm dG}) \\& \quad +3B_{\rm dG}(\textrm{I}_\dg\Psi,\textrm{I}_\dg\Psi,\Psi_{\dg},\Phi_{\rm dG})
.
\end{align*} 
This implies $\Psi_{\rm dG}$ solves \eqref{2.3.13}. The uniqueness of the solution $\Psi_{\rm dG}$ follows from the contraction result in Lemma \ref{4.3}.
\subsection{Implementation procedure}\label{algorithm}
We compute the approximate solutions $\Psi_{\rm dG }$ using Newton's method. To compute the initial guess $\Psi^0_{\rm dG }$ in Example \ref{Luo_exampple}, first solve the Oseen-Frank system 
\begin{align}\label{initial guess}
-\Delta \theta =0 \, \, \text{in } \Omega \text{ and }  \theta =g_D \,\, \text{on } \partial\Omega
\end{align}
with  the corresponding Dirichlet boundary condition, say $g_D$. The details of boundary functions corresponding to each diagonal and rotated solution can be found in Table \ref{table:initial}.
The dG formulation corresponding to \eqref{initial guess} seeks $\theta_\dg \in V_h$ such that for all $\phi_\dg \in V_h,$
\begin{align} \label{initial dg}
&\sum_{ T \in \mathcal{T}} \int_T \nabla \theta_\dg \cdot \nabla \phi_{\rm dG} \dx - \sum_{ E \in \mathcal{E}}\int_{ E} \{\frac{\partial \theta_\dg}{\partial \eta}\} [\phi_{\rm dG}] \ds +\lambda 
\sum_{E \in \mathcal{E}} \int_{E} \{\frac{\partial \phi_{\rm dG}}{\partial \eta}\} [\theta_\dg] \ds+ \sum_{E \in \mathcal{E}} \int_{E} \frac{\sigma}{h} [\theta_\dg] [ \phi_{\rm dG}] \ds \notag\\& 
= \lambda \sum_{E \in \mathcal{E}_D} \int_{E} \frac{\partial \phi_{\rm dG}}{\partial \eta} g_D \ds+ \sum_{E \in \mathcal{E}_D} \int_{E}\frac{\sigma}{h}  \phi_{\rm dG} g_D\ds. 
\end{align}
\textbf{Algorithm }\\
\begin{algorithm}[H]
	\SetAlgoLined
	Let $\mathcal{T}_0$ denote the initial triangulation \;
	\For{$ \mathcal{T}_j, j=1,2, \ldots , n$}{
		 Assemble the element matrices from \eqref{initial dg}   \;
		Solve \eqref{initial dg} using Krylov solver \;
		Compute initial guess $\Psi_{\dg }^0$ as given in \eqref{initial_guess}\;
		\While{($k< $ maximum iteratioin, error $>$ tolerance) }{
			Solve for Newton iterates $\dual{{\rm DN}_h(\Psi_{\dg }^k)\delta \Psi, \Phi_\dg}= - N_h(\Psi_{\dg }^k, \Phi_{\dg})$\;
			Update $\Psi_{\dg }^{k+1}= \Psi_{\dg }^k + \delta \Psi$
    }
Compute dG and $L^2$ norm errors and order of convergences.
	}
\label{Algorithm to find initial guess:}
\end{algorithm}
\noindent In Examples  \ref{Polynomial_example} and \ref{Annular_example}, where the exact solutions are known, the initial guesses for the Newton's method are chosen as the solutions of the corresponding linear systems.
\end{appendices}

\end{document}